\documentclass[11pt]{article}

\usepackage{color}
\definecolor{darkred}{rgb}{0.8,0,0}

\usepackage{amsmath,amsthm,amsfonts,amssymb}

\usepackage{graphicx}

\usepackage{verbatim}
\usepackage{enumerate}
\newcommand{\captionfonts}{\footnotesize}
\makeatletter  
\long\def\@makecaption#1#2{%
  \vskip\abovecaptionskip
  \sbox\@tempboxa{{\captionfonts #1: #2}}%
  \ifdim \wd\@tempboxa >\hsize
    {\captionfonts #1: #2\par}
  \else
    \hbox to\hsize{\hfil\box\@tempboxa\hfil}%
  \fi
  \vskip\belowcaptionskip}
\makeatother   

\newcommand{\nwc}{\newcommand}

\newtheorem{proposition}{Proposition}[section]
\newtheorem{lemma}[proposition]{Lemma}
\newtheorem{remark}[proposition]{Remark}
\newtheorem*{remark*}{Remark}
\newtheorem{theorem}[proposition]{Theorem}
\newtheorem{corollary}[proposition]{Corollary}
\newtheorem{definition}[proposition]{Definition}

\newtheorem{assumption}[proposition]{Assumption}
\providecommand{\U}[1]{\protect\rule{.1in}{.1in}}
\newcommand{\eps}{\varepsilon}

\newcommand{\R}{\mathbb{R}}
\newcommand{\N}{\mathbb{N}}

\nwc{\clb}{\color{blue}}
\newcommand{\ignore}[1]{}
\usepackage[utf8]{inputenc}
\usepackage{amssymb}
\usepackage{esint}
\usepackage{verbatim}
\providecommand{\U}[1]{\protect\rule{.1in}{.1in}}

\newcommand{\Ha}{\mathcal{H}}

\newcommand{\weakto}{\rightharpoonup}
\newcommand{\weakstarto}{\stackrel{\ast}{\rightharpoonup}}
\newcommand{\baru}{u}
\newcommand{\alcrit}{\alpha_{0}}
\newcommand{\embeds}{\hookrightarrow}
\DeclareMathOperator{\dist}{dist}

\DeclareMathOperator{\id}{id}
\DeclareMathOperator{\spt}{spt}

\begin{document}
\title
{A bulk--surface reaction--diffusion system for cell polarization}

\author{
B. Niethammer
\and
M. R\"oger
\and
J. J. L. Vel\'{a}zquez
}
\maketitle

\begin{abstract}
We propose a model for cell polarization as a response to an external signal which results in a system of PDEs for different variants of a protein
on the cell surface and interior respectively. We study stationary states of this model in certain parameter regimes in which several reaction rates on the
membrane as well as the diffusion coefficient within the cell are large. It turns out that in suitable scaling limits steady states converge to solutions
of some obstacle type problems. For these limiting problems we prove the onset of polarization if the total mass of protein is sufficiently small.
For some variants we can even characterize precisely the critical mass for which polarization occurs.
\\[2ex]\noindent%
\\[2ex]\noindent%
\end{abstract}

\tableofcontents

\section{Introduction}

Cell polarization in response to external chemical gradients plays a crucial role in many biological processes that involve
the motion of eukaryotic cells (cf.~\cite{RaEd17}). During the process of cell polarization the concentrations of several chemicals, some of them
attached to the cell membrane and others contained in the cytoplasm, are rearranged significantly.
One of the basic issues that corresponding models must address is the mechanism by which relatively weak chemical gradients
yield large spatial changes of the concentrations of chemicals at the cellular level.

\medskip

A significant number of the models used to describe cell polarization rely
on the local excitation, global inhibition (LEGI) mechanism which was
suggested in the seminal paper about cell polarization \cite{Mein99}.
Different forms in which the chemical pathways might be arranged in a way
consistent with this mechanism, and yield the specific chemical patterns
associated to cell polarization, are described in \cite{SkLN05}.
Typically polarization is achieved by the combination of an internal pattern forming system, a  response to an external signal that imposes some directional preference to the pattern, and the amplification of small concentration differences.

\medskip

In this paper we study a minimal model for this amplification step and examine when polarization patterns appear. These are characterized by the onset of clearly distinct regions in which the concentrations of some chemicals have different orders of magnitude. The
most remarkable feature of the proposed model is that for a suitable choice of  parameters
it is possible to approximate the model by a generalized obstacle problem. This asymptotic reduction yields a clean and analytically tractable characterization of polarized states and allows for a rigorous investigation of the onset of polarization patterns.

\medskip

The bio-chemical model that we study in this paper is a system of PDEs that is motivated by the GTPase cycle model presented in \cite{RaeRoe12,RaeRoe14}.
We consider two versions of one  protein on the cell surface which is either in an active or in an inactive state. We denote the first surface concentration as $u$ and the second as $v.$ Furthermore the inactive proteins can move to the interior of the cell and vice versa. We denote the volume concentration in the cytosol by $w$.
Our model contains three types of activation processes of the proteins which all take place on the cell membrane. First there is an intrinsic activation with  rate $a_{1}$. Second there is an
activation by  a positive feedback mechanism and a rate law given by a Michaelis-Menten law. Third, there is an activation induced by an external
(or internal) chemical signal. We assume here  that this signal has already been processed and has lead to a concentration field $c$ on the membrane of
a chemical that catalyzes the activation (the function $c$ could be also interpreted as the surface concentration of some activated receptors). For the
deactivation we again prescribe a Michaelis-Menten rate law. The use of Michaelis-Menten laws stems from the fact that the corresponding processes require some catalyzation, as the intrinsic activation and deactivation of GTPase proteins is typically very slow \cite{BoRW07}.

\medskip

To introduce our mathematical model, let $\Omega\subset \R^3$ denote the cell and $\Gamma:=\partial \Omega$ the cell membrane.
The  assumptions described above give rise to the following bulk-surface reaction diffusion system.
\begin{align}
	 \partial _{t}u  &  =\Delta u+\left(  a_{1}+\frac{a_{2}u}{a_{3}%
+u}+a_7c\right)  v-\frac{a_{4}u}{1+u}\qquad&\text{ on }\Gamma\times (0,T),\label{eq:org-1}\\
	 \partial _{t}v  &  =\Delta v-\left(  a_{1}+\frac{a_{2}u}{a_{3}%
+u}+a_7c\right)  v+\frac{a_{4}u}{1+u}-a_{5}v+a_{6}w &\text{ on }\Gamma\times (0,T),\label{eq:org-2}\\
	 \partial _{t}w  &  =D\Delta w\qquad&\text{ in } \Omega\times (0,T),\label{eq:org-3}\\
-D\frac{\partial w}{\partial n}  &  =-a_{5}v+a_{6}w\qquad&\text{ on }\Gamma\times (0,T). \label{eq:org-4}%
\end{align}
Here $c\colon \Gamma \times (0,T)\to \R$ is the (processed) external stimulus, with some abuse of notation  $\Delta u$ and $\Delta v$ denote the Laplace-Beltrami operator on the surface $\Gamma$, while $\Delta w$ is the usual Laplacian.

We complement the system with initial conditions
\begin{equation}
	w(\cdot,0) = w_0\quad\text{ in }\Omega,\quad u(\cdot,0)\,=\, u_0\quad\text{ on }\Gamma,\quad v(\cdot,0)\,=\, v_0\quad\text{ on }\Gamma \,,\label{eq:org-initial}
\end{equation}
where $w_0\colon \Omega\to \R$ and $u_0,v_0\colon \Gamma\to\R$ are given nonnegative data.

Solutions of \eqref{eq:org-1}-\eqref{eq:org-initial} satisfy the mass conservation property
\begin{equation}
	\int_\Omega w(x,t)\,dx + \int_\Gamma (u(y,t)+v(y,t))\,d\Ha^2(y) \,=\, \int_\Omega w_0\,dx + \int_\Gamma (u_0+v_0)\,d\Ha^2 \quad\text{ for all }t\geq 0. \label{eq:mass_conservation}
\end{equation}

\medskip

In this paper we will mainly study for given $c=c(x)$ the stationary solutions of \eqref{eq:org-1}-\eqref{eq:org-4}.
Polarization patterns arise here under the assumption that several of the
reaction rate parameters $a_{1},a_{2},...,a_7$, the diffusion coefficient $D$ and the total mass of proteins are large. This
will be parametrized by a large parameter $\varepsilon^{-1}>0$. The most remarkable feature of the patterns is that a suitably rescaled version of $u$ converges as $\varepsilon \rightarrow 0$ to the solution of a variational inequality which is closely related to the classical obstacle problem \cite{KiSt00}. Responsible for this is the presence of the
inhibitory Michaelis-Menten reaction term $\frac{u}{1+u}$, whereas the intrinsic activation term $a_{1}v$ and the positive feedback activation term $\frac{a_2 uv}{a_3+u}$ do not contribute to the limit.
Let us give a simple heuristic explanation for the impact of the Michaelis-Menten reaction term $R(u) =\frac{u}{1+u}$.
In the parameter regimes that we consider it will be convenient to introduce the rescaled concentration $U=\varepsilon u$ in order to account for the large values of $u$ in some
regions. Then, the function $R\big( \frac{U}{\varepsilon }\big) $ converges in a suitable
sense to the maximal monotone graph $\xi (U) $ such that $\xi (0) =[0,1]$, $\xi (U) =1$ if $U>0$.
It is well known that several variational inequalities, for instance the one associated
to the obstacle problem, can be reformulated in terms of maximal monotone
graphs \cite{Br73}, \cite[Sec.2.2]{Diaz85}, compare the Remark \ref{rem:maxmon} below.

\medskip

Actually, in this paper we will consider two different types of scaling limits for stationary solutions of \eqref{eq:org-1}-\eqref{eq:org-4}.
In the first one we will assume that $D=\infty$ before taking the limit $\varepsilon \to 0$, which  is motivated by the fact that the cytosolic diffusion is typically much larger than lateral diffusion over the membrane \cite{KhHW00}.
Then the stationary version of equation \eqref{eq:org-3} becomes just the
formula $w=M-\int_{\Gamma }(u+v)$, where $M$ is the total
amount of protein, and the system reduces to two coupled surface PDEs that include, as a remnant of the bulk-surface coupling, a nonlocal term. Taking then the limit $\varepsilon\rightarrow 0$ we  obtain that the limit satisfies a variational inequality for a suitable PDE.
In the second scaling limit we will take $D$ large but finite (finite
cytosolic diffusion case). The limit $\varepsilon \rightarrow 0$ will then
yield a variational inequality for an operator which contains, in addition
to partial derivatives, the concatenation of the Neumann to Dirichlet map (cf. Appendix \ref{app:ND}) and the Laplace-Beltrami operator.

\medskip

One particular advantage of the limiting obstacle-type problems is that they allow for an easy characterization of polarized states:
these are described by the property that the rescaled
concentration of the active proteins takes the value zero on a set of
positive measure and it takes a positive value on the complement, which has
also positive measure. We will prove in this paper the onset of polarization
patterns, for both cases, if the total (rescaled) mass of protein $m$ is
sufficiently small. In addition, we show that patterns are localized in regions, where the signal $c$ is large.
We will also prove that in the case $D=\infty $, for given concentration $c$, there exists a critical mass of protein $m_{\ast }$ such
that polarization takes place for $m<m_{\ast }$ and it does not for $m>m_{\ast }$. In the case $D<\infty$ due to the presence of the nonlocal contributions,
it is not clear if the critical mass $m_{\ast }$
exists with the same degree of generality concerning the domain $\Omega $
and the concentration $c$. However, in the case of spherical domains $\Omega$ the nonlocal operator reduces to the Dirichlet to Neumann operator
(cf. Appendix \ref{app:ND}) and we will be able
to prove the existence of the critical mass $m_{\ast }$ if $D<\infty$ for
these particular domains and general concentrations $c$.

\medskip

The model \eqref{eq:org-1}-\eqref{eq:org-4} is different from LEGI-type models and rather
describes the signal amplification following a first polarization of the cell, expressed by a heterogeneous distribution $c$. In
the scaling limits under consideration solutions of \eqref{eq:org-1}-\eqref{eq:org-4} do not exhibit spontaneous patterns if $c$ is constant unlike the classical Gierer-Meinhardt models (cf.~\cite{GiMe72}).

\medskip

The system \eqref{eq:org-1}-\eqref{eq:org-initial} is closely related to models for spontaneous cell polarization (in absence of an external signal)
considered in \cite{RaeRoe14,HaRo18}. There, more general reaction kinetics but spatially homogeneous rate constants are assumed.
The model \eqref{eq:org-1}-\eqref{eq:org-initial} belongs to the class of bulk-surface partial differential equations that appears in a variety of different applications and has attracted quite some attention over the last years, see for example \cite{LeRa05,MPCV16,RaEd17,DMMS18} and the references therein for applications to cell biology and \cite{MoSh16,BKMS17,ElRV17,FeLT18,HaRo18} for recent well-posedness results and rigorous asymptotic limits. In \cite{ElRV17}, for a different bulk-surface system where the nonlinearity is contained in the Robin boundary condition, a fast-reaction limit was derived that also takes the form of an obstacle-type problem involving the Dirichlet to Neumann map. However, compared to our contribution the asymptotic analysis and the limiting obstacle-type problems are different.
The convergence of elliptic PDEs with suitably rescaled Michaelis-Menten reactions to a classical obstacle model has already been observed in \cite{BrNi82}.
There the motivation was an approximation of the obstacle problem by a bounded penalty method suitable for numerical simulations. No investigation of pattern forming properties in the sense of our analysis has been done in \cite{BrNi82,ElRV17}. For an investigation of axially symmetric cap-like patterns in a related but different model see \cite{DMMS18}.

\medskip

The plan of the paper is the following. In Section \ref{S.wellposedness} we will briefly
review the well-posedness of the initial value problem \eqref{eq:org-1}-\eqref{eq:org-initial} and
establish the existence of stationary states (cf. \eqref{eq:org-stat-1}-\eqref{eq:org-stat-5}).
In Section \ref{S.Dinfty} we will first investigate the limit of infinite cytosolic
diffusion $D\rightarrow \infty$ for stationary solutions. In this limit $w$ converges to a constant
and the model \eqref{eq:org-1}-\eqref{eq:org-initial} becomes a nonlocal elliptic system (cf.~\eqref{eq:AtDInf1-0}-\eqref{eq:AtDInf3-0}) containing the nonlocal term $\int_{\Gamma }\left( u+v\right)$. Next, we prove that in some suitable scaling limit the solution $u$ of the rescaled
system \eqref{eq:AtDInf1}-\eqref{eq:AtDInf3} converges to the unique solution of an obstacle problem
for the Laplace operator (cf. Theorem \ref{Limeps}). In Section \ref{Ss.criticalmass} we derive precise conditions for the onset of polarization and characterize the positivity set of $u$ for small mass in Section \ref{Ss.localization}. In Section \ref{S.Dfinite} we consider the analogous problems for finite cytosolic diffusion coefficients $D$. We derive in a scaling limit analogous to the case $D=\infty $ a generalized obstacle problem in Theorem \ref{thm:obs1} containing the Dirichlet to Neumann operator. We can prove also prove polarization for small mass for the
resulting model (cf. Subsection \ref{Ss.localization-1}). We obtain a more precise
description of the localization property, as well as the existence of a critical mass in the case of spherical domains $\Omega $ in Subsections \ref{Ss.existenceobs} and \ref{Ss.critmass2}.

\bigskip

\section{Well-posedness and existence of steady states}
\label{S.wellposedness}

Let us state the main assumptions that we impose in the following.
\begin{assumption}\label{ass:pbm2}
Let $\Omega\subset \R^3$ be an open, bounded, connected set with $C^3$-regular boundary $\Gamma=\partial \Omega$. Assume $a_1,a_2\geq 0$, $a_3,a_4,a_5,a_6,a_7>0$ and $D\geq 1$.
Furthermore we assume that $c \colon \Gamma \to \R_+$ is continuous and strictly positive. This in particular implies that there exists $c_0>0$ such that
\begin{equation}
c(x)  \geq c_{0}>0 \quad \mbox{ for all } x\in \Gamma\,. 
	\label{PosCond}%
\end{equation}
\end{assumption}

\begin{remark}
Throughout the paper we will denote by $W^{k,p}(\Omega)$ the usual Sobolev spaces over $\Omega$ and by $W^{k,p}(\Gamma)$ the corresponding Sobolev spaces on the manifold $\Gamma$. For $T>0$ we set $\Omega_T:= \Omega\times (0,T)$ and $\Gamma_T:=\Gamma\times (0,T)$. For a Banach space $X$ we consider the Bochner spaces $L^2(0,T;X)$ and the spaces $H^1(0,T;X)$ of functions in $L^2(0,T;X)$ with weak time derivative in $L^2(0,T;X)$.

The relevant diffusion operator on $\Gamma$ is the corresponding Laplace-Beltrami operator, see for example \cite{PrSi16}. For simplicity we just write $\Delta$ if there is no reason for confusion. We remark that classical elliptic $L^p$- and Schauder-regularity results \cite[Sec.~6.1, Sec.~9.5]{GiTr01} and a partition of unity argument yield the corresponding $L^p$-regularity properties for elliptic equations involving the Laplace--Beltrami operator on $\Gamma$. In fact, in local coordinates the Laplace--Beltrami operator corresponds to an elliptic operator in divergence form (with $C^2$-regular coefficients in our case). Similarly one deduces parabolic $L^2$-regularity in analogy to \cite[Theorem 7.1.5]{Evan10} for evolution problems on $\Gamma$ involving the Laplace--Beltrami operator.

For simplicity we will often neglect $dx$ and $d\Ha^2$ in the corresponding integrals if there is no reason for confusion.

For a function $\varphi:\Gamma\to\R$ we let $\overline\varphi:= \fint_\Gamma \varphi=\frac{1}{|\Gamma|}\int_\Gamma \varphi$.
\end{remark}

We first quote an existence and uniqueness result for the full evolution problem that was shown in  \cite{HaRo18}, Proposition 3.3 (there uniformly bounded initial data were assumed and only the case of constant $c$ is covered, but inspecting the proof we see that the result remains valid under the present assumptions).
\begin{theorem}\label{thm:existence1}
Let $T>0$ and let nonnegative initial data $u_0,v_0 \in L^2(\Gamma)$, $w_0\in L^2(\Omega)$ be given with $\int_{\Gamma} \big(u_0+v_0\big) + \int_{\Omega} w_0=M$.
Under Assumption \ref{ass:pbm2} the system \eqref{eq:org-1}-\eqref{eq:org-initial} has a unique nonnegative weak solution $(u,v,w)$ with
$u,v \in L^2(0,T;H^1(\Gamma))\cap H^1(0,T;H^1(\Gamma)^*)$ and $w\in L^2(0,T;H^1(\Omega))\cap H^1(0,T;H^1(\Omega)^*)$.
\end{theorem}

The following a-priori estimate will be useful to prove the existence of steady states as well as to pass to the limit $D\to \infty$.
\begin{proposition}\label{prop:invar}
Assume that $(w,u,v)$ is a nonnegative solution of \eqref{eq:org-1}-\eqref{eq:org-initial} with $\int_{\Gamma} \big(u_0+v_0\big) + \int_{\Omega} w_0=M$.
Define $L:[0,T]\to\R$ by
\begin{equation*}
	L(t) := \int_\Omega \frac{a_6}{2} w^2(x,t)\,dx + \int_\Gamma \frac{1}{2}(u^2+a_5 v^2)(y,t)\,d\Ha^2(y).
\end{equation*}
There exist constants $C_1,C_2$ only depending on $\Omega$, $c$ and $a_1,\dots,a_7$, and for any $\sigma>0$ a constant $C_3$ that in addition depends on $\sigma$, such that for $L(0)\leq C_1(M+M^2)$ the following properties hold:
\begin{align}
	&L(t)\leq C_1(M+M^2)\quad\text{ for any }t\in [0,T], \label{eq:invar}\\
	&\|w\|_{L^\infty(0,T;L^2(\Omega))} + \sqrt{D}\|\nabla w\|_{L^2(\Omega_T)} + \frac{1}{\sqrt D}\|\partial_t w\|_{L^2(0,T;H^1(\Omega)^*)} \notag\\
	&\qquad\qquad + \|u,v\|_{L^\infty(0,T;L^2(\Gamma))\cap L^2(0,T;H^1(\Gamma))} + \|\partial_t u,\partial_t v\|_{L^2(0,T;H^1(\Gamma)^*)} \leq C_2(1+M), \label{eq:invar-2}\\
	&\|w\|_{H^1(\sigma,T;L^2(\Omega))} + \|w\|_{L^2(\sigma,T;H^2(\Omega))} + \|u,v\|_{H^1(\sigma,T;L^2(\Gamma))} + \|u,v\|_{L^2(\sigma,T;H^2(\Gamma_T))} \leq C_3(1+M). \label{eq:invar-3}
\end{align}
\end{proposition}
\begin{proof}
We test \eqref{eq:org-1}-\eqref{eq:org-3} with $u,v$ and $w$, respectively, to obtain
\begin{align}
	&\frac{d}{dt}L(t) 
	+\int_\Omega a_6 D|\nabla w|^2 +\int_\Gamma \big(|\nabla u|^2+a_5|\nabla v|^2\big) \notag\\
	= &\int_\Gamma \Big( (a_6w-a_5v) (a_5v-a_6w) + (u-a_5v)\big((a_1+\frac{a_2 u}{a_3+u}+a_7c)v-\frac{a_4u}{1+u}\big)\Big) \notag\\
	\leq & C\int_\Gamma (v+u^2+v^2) \notag\\
	\leq & -L(t) + C\int_\Gamma (v+u^2+v^2) + \frac{a_6}{2}\int_\Omega w^2, \label{eq:invar-1}
\end{align}
where here and in the following $C$ only depends on $\Omega$, $c$ and $a_1,\dots,a_7$, unless additional dependencies are indicated.
We now use interpolation estimates, Sobolev embedding and Poincar{\'e} inequality to obtain, with $\overline{u} = \fint_\Gamma u$
\begin{align*}
	\int_\Gamma  u^2 &= \int_\Gamma u(u-\overline{u}) + \|u\|_{L^1(\Gamma)}^2 \\
	&\leq  \|u-\overline{u}\|_{L^2(\Gamma)}\|u\|_{L^2(\Gamma)}+ \|u\|_{L^1(\Gamma)}^2\nonumber\\
	&\leq \|u-\overline{u}\|_{L^4(\Gamma)}^{\frac{2}{3}}\|u-\overline{u}\|_{L^1(\Gamma)}^{\frac{1}{3}}\|u\|_{L^2(\Gamma)}+ \|u\|_{L^1(\Gamma)}^2\nonumber\\
	&\leq C\|\nabla u\|_{L^2(\Gamma)}^{\frac{2}{3}}\|u-\overline{u}\|_{L^1(\Gamma)}^{\frac{1}{3}}\|u\|_{L^2(\Gamma)}+ \|u\|_{L^1(\Gamma)}^2\nonumber\\
	&\leq \frac{\delta_1}{2}\|\nabla u\|_{L^2(\Gamma)}^2 + \frac{1}{2}\|u\|_{L^2(\Gamma)}^2 + C(\delta_1)\|u\|_{L^1(\Gamma)}^2.
\end{align*}
Together with the mass conservation property \eqref{eq:mass_conservation} and analogous calculations for $v,w$ we deduce that
\begin{align*}
	\int_\Gamma u^2 &\leq \delta_1 \|\nabla u\|_{L^2(\Gamma)}^2 + C(\delta_1)M^2,\quad
	\int_\Gamma v^2 \leq \delta_2 \|\nabla v\|_{L^2(\Gamma)}^2 + C(\delta_2)M^2,\\
	\int_\Omega w^2 &\leq \delta_3 \|\nabla w\|_{L^2(\Omega)}^2 + C(\delta_3)M^2.
\end{align*}
Using this in \eqref{eq:invar-1} with appropriate choices of $\delta_1,\delta_2,\delta_3$ yields
\begin{equation}
	\frac{d}{dt}L(t) + \frac{1}{2}\int_\Omega a_6 D|\nabla w|^2 +\frac{1}{2}\int_\Gamma \big(|\nabla u|^2+a_5|\nabla v|^2\big) \leq  -L(t) + C_1(M+M^2). \label{eq:dtL}
\end{equation}
Since $L(0)\leq C_1(M+M^2)$ this implies by an invariant region principle \eqref{eq:invar} and the $L^\infty(L^2)$ bounds in \eqref{eq:invar-2}. We also deduce from \eqref{eq:dtL} the $L^2(L^2)$ bounds for the gradients. Then, using the equations, this yields the required bounds for the time derivatives in \eqref{eq:invar-2}.

To prove \eqref{eq:invar-3} we fix $\sigma>0$ and a cut-off function $\varphi\in C^\infty(\R)$ with $\varphi=0$ on $[0,\frac{\sigma}{2}]$, $\varphi=1$ on $[\sigma,\infty)$ and $0\leq \varphi\leq 1$, $0\leq \varphi^\prime\leq \frac{4}{\sigma}$.
Then the functions $\tilde u(\cdot,t)=\varphi(t) u(\cdot,t)$, $\tilde v(\cdot,t)=\varphi(t) v(\cdot,t)$, $\tilde w(\cdot,t)=\varphi(t) w(\cdot,t)$ are solutions of
\begin{align}
	\partial _{t}\tilde u  -\Delta \tilde u &  = \tilde f + \varphi^\prime u\qquad\text{ on }\Gamma\times (0,T)\label{eq:tilde-1}\\
	\partial _{t}\tilde v  -\Delta \tilde v  &  =-\tilde f -a_5 \tilde v +a_6\tilde w + \varphi^\prime v\qquad\text{ on }\Gamma\times (0,T)\label{eq:tilde-2}\\
	\partial _{t} \tilde w  -D\Delta \tilde w  &  = \varphi^\prime w\qquad\text{ in } \Omega\times (0,T)\label{eq:tilde-3}\\
-D\frac{\partial \tilde w}{\partial n}  &  = a_5 \tilde v -a_6\tilde w\qquad\text{ on }\Gamma\times (0,T)\label{eq:tilde-4}%
\end{align}
with zero initial conditions and with $\tilde f=\big(  a_{1}+\frac{a_{2}u}{a_{3}%
+u}+a_7c\big)\tilde v-\frac{a_{4}\tilde u}{1+u}$. We deduce from \eqref{eq:invar-2} that  the right-hand sides of \eqref{eq:tilde-1}, \eqref{eq:tilde-2}, \eqref{eq:tilde-4} are all bounded in $L^2(\Gamma_T)$ and that the right-hand side of \eqref{eq:tilde-3} is bounded in $L^2(\Omega_T)$ by constants that only depends on $M,\sigma,T$, $\Omega$, $c$ and $a_1,\dots,a_7$. By $L^2$-regularity theory we obtain $L^2(0,T;H^2(\Gamma))\cap H^1(0,T;L^2(\Gamma))$-bounds for $\tilde u,\tilde v$ and then an $L^2(0,T;H^2(\Omega))\cap H^1(0,T;L^2(\Omega))$-bound for $\tilde w$ that only depends on $M,\sigma,T$, $\Omega$, $c$ and $a_1,\dots,a_7$. Since $\varphi=1$ on $[\sigma,T]$ the estimates for $\tilde u,\tilde v,\tilde w$ yield \eqref{eq:invar-3}.
\end{proof}

We next prove the existence of stationary states using Theorem \ref{thm:existence1} and the a-priori estimates from Proposition \ref{prop:invar}.

\begin{theorem}\label{thm:stat-states-Dpositive}
For any $M>0$ there exists a nonnegative solution $(u,v,w)$, such that $u,v \in C^2(\Gamma)$, $w\in  C^2(\overline{\Omega})$, of the system
\begin{align}
	0  &  =\Delta u+\left(  a_{1}+\frac{a_{2}u}{a_{3}%
+u}+a_7c\right)  v-\frac{a_{4}u}{1+u}\qquad&&\textrm{ on }\Gamma\label{eq:org-stat-1}\\
	0  &  =\Delta v-\left(  a_{1}+\frac{a_{2}u}{a_{3}%
+u}+a_7c\right)  v+\frac{a_{4}u}{1+u}-a_{5}v+a_{6}w\qquad&&\textrm{ on }\Gamma\label{eq:org-stat-2}\\
	0  &  =D\Delta w\qquad&&\textrm{ in } \Omega\label{eq:org-stat-3}\\
	-D\frac{\partial w}{\partial n}  &  =-a_{5}v+a_{6}w\qquad&&\textrm{ on }\Gamma\label{eq:org-stat-4}\\
M&=\int_\Omega w + \int_\Gamma (u+v)\label{eq:org-stat-5}.
\end{align}
Moreover,
 \begin{equation}\label{eq:l2apriori-ss}
  \|u,v\|_{H^1(\Gamma)}^2 + \|w\|_{L^2(\Gamma)}^2 + D\|\nabla w\|_{L^2(\Omega)}^2 \leq C(M)\,.
 \end{equation}
\end{theorem}
\begin{proof}
Denote by $X$ the set of all $(u_0,v_0,w_0)\in L^2(\Gamma)^2\times L^2(\Omega)$ such that
\begin{align*}
	\int_\Omega w_0 + \int_\Gamma (u_0+v_0)= M \qquad\text{ and }\qquad\int_\Omega \frac{a_6}{2} w_0^2 + \int_\Gamma \frac{1}{2}(u_0^2+a_5 v_0^2) \leq C_1(M+M^2),
\end{align*}
with $C_1$ from Proposition \ref{prop:invar}.
Fix $T\geq 3$. By Theorem \ref{thm:existence1} for any $(u_0,v_0,w_0)\in X$ there exists a unique weak solution $(u,v,w)$ to the evolution problem in $(0,T)$.
By Proposition \ref{prop:invar} and by the mass conservation property \eqref{eq:mass_conservation} we deduce that $(u(t),v(t),w(t))\in X$ for all $t\in (0,T)$. Moreover, the bounds
\eqref{eq:invar-2}, \eqref{eq:invar-3} are satisfied for any $\sigma>0$. In particular (see e.g. \cite{Evan10}, Theorem 3, Section 5.9)  $w\in C^0([0,T];L^2(\Omega))$ and $u,v \in C^0([0,T];L^2(\Gamma))$. Similarly, for any $\sigma>0$ we deduce $w\in C^0([\sigma,T];H^1(\Omega))$ and $u,v \in C^0([\sigma,T];H^1(\Gamma))$ with bounds that depend only on $\sigma$ and the data.

In particular, for any $\tau\in (0,1)$ the map
\begin{equation*}
	S(\tau): X\to X,\quad  (u_0,v_0,w_0)\mapsto (u(\tau),v(\tau),w(\tau))
\end{equation*}
is well-defined, and by Sobolev embedding compact. Considering differences of two solutions with initial data $(w_0,u_0,v_0)$ and $(\tilde w_0,\tilde u_0,\tilde v_0)$, respectively, we obtain by similar calculations as in Proposition \ref{prop:invar} that
\begin{equation*}
	\|S(\tau)(u_0,v_0,w_0)-S(\tau)(\tilde u_0,\tilde v_0,\tilde w_0)\|_{L^2(\Gamma)^2\times L^2(\Omega)} \leq C\|(u_0,v_0,w_0)-(\tilde u_0,\tilde v_0,\tilde w_0)\|_{L^2(\Gamma)^2\times L^2(\Omega)}.
\end{equation*}
Therefore $S(\tau)$ is continuous for any $0<\tau<T$.

By Schauders fixed point theorem for any $\tau\in (0,1)$, $\tau=1/N$, there exists a fixed point $(u^\tau,v^\tau,w^\tau)\in X$ of $S(\tau)$. Since the system  \eqref{eq:org-1}-\eqref{eq:org-4} is autonomous and solutions are unique we obtain a solution $(u_\tau,v_\tau,w_\tau)$ of this system on $(0,T)$ with the property that
\begin{equation*}
	(u_\tau(k\tau),v_\tau(k\tau),w_\tau(k\tau)) = (u^\tau,v^\tau,w^\tau) \quad\text{ for all }k\in\N.
\end{equation*}

By Proposition \ref{prop:invar} we obtain for any $N\in\N$ uniform bounds
\begin{equation*}
	\|w_\tau \|_{H^1((1,2);L^2(\Omega))\cap L^2((1,2);H^2(\Omega))} + \|u_\tau,v_\tau\|_{H^1((1,2);L^2(\Gamma))\cap L^2((1,2);H^2(\Gamma))} \leq C_3(1+M).
\end{equation*}
Therefore there exists a subsequence $N\to\infty$ (not relabeled), hence $\tau=\tau(N)\to 0$, such that
\begin{align*}
	&w_\tau\weakto w\text{ in }H^1((1,2);L^2(\Omega))\cap L^2((1,2);H^2(\Omega)),\\
	&u_\tau\weakto u\text{ in }H^1((1,2);L^2(\Gamma))\cap L^2((1,2);H^2(\Gamma)),\\
	&v_\tau\weakto v\text{ in }H^1((1,2);L^2(\Gamma))\cap L^2((1,2);H^2(\Gamma)).
\end{align*}
It is easy to show that $(u,v,w)$ is a solution to \eqref{eq:org-1}-\eqref{eq:org-4} on the time interval $(1,2)$. Moreover, by the Aubin-Lions lemma, we have strong $L^2(\Gamma\times (1,2))^2\times L^2(\Omega\times (1,2))$-convergence of $(u_\tau,v_\tau,w_\tau)$ to $(u,v,w)$ and for a set $R\subset [1,2]$ of full measure and all $t\in R$ strong convergence in $L^2(\Gamma)^2\times L^2(\Omega)$ of $(u_\tau(t),v_\tau(t),w_\tau(t))$ to $(u(t),v(t),w(t))$.

Next, since by \cite{Evan10}, Chapter 5.9,  we have uniform bounds in $C^0([1,2];H^1(\Gamma)^2\times H^1(\Omega))$ we in addition have $L^2$-convergence of $(u^\tau,v^\tau,w^\tau)=(u_\tau(1),v_\tau(1),w_\tau(1))$ to a limit $(u_*,v_*,w_*)\in X\cap H^1(\Gamma)^2\times H^1(\Omega)$.

Now choose $t\in R$ arbitrary. Then there exists a sequence $t_N$ such that $t_N\in \N\tau$ and $t_N\to t$.  We then compute that in $L^2(\Gamma)$
\begin{align*}
	\| u(t)- u_*\| &\leq \|u(t)-u_\tau(t)\| + \|u_\tau(t)-u_\tau(t_N)\| + \|u_\tau(t_N)-u_*\| \\
	&\leq \|u(t)-u_\tau(t)\| + C|t_N-t|^{\frac{1}{2}}\|u_\tau\|_{C^{\frac{1}{2}}([1,2];L^2(\Omega))} + \|u^\tau-u_*\|\,\to\, 0,
\end{align*}
where we have used that $u_\tau$ is uniformly bounded in $H^1((1,2);L^2(\Omega))\hookrightarrow  C^{\frac{1}{2}}([1,2];L^2(\Omega))$. This shows $u(t)=u_*$ for all $t\in R$. Similarly we deduce $v(t)=v_*$ and $w(t)=w_*$ for all $t\in R$. Since $(u,v,w)\in C^0([1,2];H^1(\Gamma))^2\times C^0([1,2];H^1(\Omega))$ we have $(u,v,w)=(u_*,v_*,w_*)$ on $[1,2]$.

We therefore have obtained a time-independent solution of \eqref{eq:org-1}-\eqref{eq:org-4} on $(1,2)$ that satisfies the mass constraint. This gives the required solution of the stationary system.

Starting from the $H^1$ regularity of the solution and the Sobolev embeddings $H^1(\Gamma)\hookrightarrow L^p(\Gamma)$ for any $1\leq p<\infty$, and $H^1(\Omega) \hookrightarrow L^4(\Gamma)$ we deduce from classical regularity that $u\in W^{2,p}(\Gamma)$, $v\in W^{2,4}(\Gamma)$. Since $W^{2,4}(\Gamma)\hookrightarrow C^{1,\frac{1}{2}}(\Gamma)$ we deduce from \cite[Theorem 6.31]{GiTr01} that $w\in C^{2,\frac{1}{2}}(\overline{\Omega})$. Standard Schauder regularity then implies higher regularity for $v,u$.

The estimate \eqref{eq:l2apriori-ss} is obtained from $(u,v,w)\in X$ and \eqref{eq:invar-2}.
\end{proof}

\section{The limit of infinite cytosolic diffusivity: $D=\infty$}\label{S.Dinfty}
\subsection{Passage to the limit $D\to \infty$}
In biological cells the diffusion in the cytosol is typically much faster than the lateral diffusion on the membrane. This means in our rescaling that $D\gg 1$ and motivates to consider an asymptotic reduction $D\to\infty$. We obtain a system of two surface PDEs for $u,v$, whereas the variable $w$ becomes spatially constant and is determined by the prescribed total amount of protein in the cell. This introduces a nonlocal term in the PDE system for $u,v$.
\begin{theorem}\label{thm:stat-states-D=infty}
For any $M>0$ there exist $u,v \in C^2(\Gamma)$, and a nonnegative constant $w$, such that
\begin{align}
0  &  =\Delta u+\left(  a_{1}+\frac{a_{2}u}{a_{3}+u}+a_7c\right)  v-\frac{a_{4}u}{1+u}\qquad \mbox{ on } \Gamma\,, \label{eq:AtDInf1-0}\\
0  &  =\Delta v-\left(  a_{1}+\frac{a_{2}u}{a_{3}+u}+a_7c\right)  v+\frac{a_{4}u}{1+u}-a_5v+a_6 w\qquad \mbox{ on } \Gamma\,,\label{eq:AtDInf2-0}\\
w &  =M-\int_{\Gamma}(u+v)\,.
\label{eq:AtDInf3-0}%
\end{align}
\end{theorem}
\begin{proof}
 Consider a sequence $(D_k)_k$ with $D_k\to\infty$ as $k\to\infty$ and let $(u_k,v_k,w_k)$ be a solution to \eqref{eq:org-stat-1}-\eqref{eq:org-stat-5} with $D$ replaced by $D_k$.
 By the bound \eqref{eq:l2apriori-ss} and standard compactness arguments we deduce the existence of a subsequence (not relabeled) $D_k\to\infty$ such that $(u_k,v_k,w_k)$ converges weakly in $H^1(\Gamma)^2\times H^1(\Omega)$ to some limit $(u,v,w)$. Moreover, by the Rellich theorem and the compactness of the trace mapping $H^1(\Omega)\to L^2(\Gamma)$ we have strong convergence in $L^2(\Gamma)^2\times L^2(\Omega)$.
 This allows to pass to the limit in the equations \eqref{eq:org-stat-1}, \eqref{eq:org-stat-2} and \eqref{eq:org-stat-5} and to deduce \eqref{eq:AtDInf1-0}-\eqref{eq:AtDInf3-0}. The higher regularity of the solution is deduced from standard regularity results as above.
 \end{proof}

\subsection{Derivation of the obstacle problem}

Our goal in this section is to consider a suitable scaling limit of the system \eqref{eq:AtDInf1-0}-\eqref{eq:AtDInf3-0}.
More precisely, we introduce the following rescalings

\begin{equation}
a_{4}\leadsto\frac{a_{4}}{\varepsilon}\,,\quad a_{5}\leadsto\frac{a_{5}}{\varepsilon}\,,\quad  a_{6}\leadsto\frac{a_{6}%
}{\varepsilon}\,,\quad  a_{7}\leadsto\frac{a_{7}}{\varepsilon} \label{E2}%
\end{equation}
and assume that the coefficients $a_{1},a_{2},a_3$ are kept of order one and that $a_4,a_5,a_6,a_7>0$. Upon redefining $c$, without loss of generality we can assume that $a_7=1$.
Moreover, in order to obtain a nontrivial limit
we also need to assume that the total number of proteins $M$
rescales like $\frac{1}{\varepsilon}$, more precisely $M=\frac{m}{\varepsilon}$ for some fixed $m>0$. Denoting by $(u_\eps,v_\eps,w_\eps)$ the solution of the rescaled system \eqref{eq:AtDInf1-0}-\eqref{eq:AtDInf3-0} an asymptotic expansion shows that we should expect that $U_\eps:=\eps u_\eps$, $v_\eps$, $w_\eps$ are of order one.
We therefore rewrite the rescaled system as
\begin{align}
  0  &  =\Delta U_\eps+\left(  \eps a_{1}+ \frac{\eps a_{2}U_\eps}{\eps a_{3}+U_\eps}+c\right)  v_\eps-\frac{a_{4}U_\eps}{\eps +U_\eps}\,\qquad \mbox{ on }\Gamma\,, \label{eq:AtDInf1}\\
  0  &  = \eps \Delta v_\eps-\left(  \eps a_{1}+ \frac{\eps a_{2}U_\eps}{\eps a_{3}+U_\eps}+c\right)  v_\eps+\frac{a_{4}U_\eps}{\eps +U_\eps}-a_5v_\eps+a_6 w_\eps  \qquad \mbox{ on } \Gamma\,, \label{eq:AtDInf2}\\
  \eps w_\eps &  =m-\int_{\Gamma}\big(U_\eps+\eps v_\eps\big)\,. \label{eq:AtDInf3}%
\end{align}
We continue to assume that $w_\eps\geq 0$, hence $\int_{\Gamma} \big(U_\eps+\eps v_\eps\big) \leq m$.

For convenience we state our results mainly in terms of a function $g:\Gamma\to (0,1)$ instead of $c$,
\begin{equation}
  g(x)  =\frac{c(x)}{c(x)+a_{5}}, \quad \ x\in\Gamma. \label{defg}%
\end{equation}

\begin{theorem}
\label{Limeps}
Suppose that $\left\{(U_{\varepsilon},v_{\varepsilon},w_{\eps})\right\}_{\varepsilon>0}$ is a
family of nonnegative solutions of (\ref{eq:AtDInf1})-(\ref{eq:AtDInf3}) and fix an arbitrary sequence $(\varepsilon_{j})_{j\in\N}$ with $\lim_{j\rightarrow\infty}\varepsilon_{j}=0$.

Then, there exist a subsequence $\eps_j\to 0$ (not relabeled), $\baru\in W^{2,p}(\Gamma)$ for any $1\leq p<\infty$, and a measurable function $\xi\colon\Gamma\rightarrow [0,1]$ such that with $j\to\infty$
\begin{align*}
	U_{\varepsilon_{j}}  &\to  \baru \qquad \text{ in }L^1(\Gamma),\\
	\frac{U_{\varepsilon_{j}}}{\varepsilon_{j}+U_{\varepsilon_{j}}}  &  \overset{*}{\rightharpoonup} \xi\qquad  \text{ weakly* in }L^{\infty}(\Gamma),
\end{align*}
and such that $\xi u = u$ almost everywhere on $\Gamma$.

Moreover, there exists $\alpha\in\R$ such that the functions $\baru$ and $\xi$ satisfy almost everywhere on $\Gamma$
\begin{equation}
	-\Delta\baru=\alpha g-a_{4}(1-g) \xi \label{B4}%
\end{equation}
and we have
\begin{equation}
  \int_{\Gamma}\baru=m. \label{massU}%
\end{equation}
\end{theorem}

\begin{proof}
Integrating (\ref{eq:AtDInf1}) and  (\ref{eq:AtDInf2}) over $\Gamma$ we obtain
\begin{align}
\int_{\Gamma}\left(  a_{1}\varepsilon+\frac{a_{2}\varepsilon U_\eps}%
{a_{3}\varepsilon+U_\eps}+c\right) v_\eps  &  = \int_{\Gamma}\frac{a_{4}U_\eps}{\varepsilon+U_\eps}\label{A1}\\
\int_{\Gamma}a_{5}v_\eps  &  = a_6 w_\eps |\Gamma| \label{A2}%
\end{align}
Then, using that $c\geq c_{0}>0,\ U_\eps\geq0,\ v_\eps\geq0$ and $\frac{U_\eps}{\varepsilon+U_\eps}\leq1$ in (\ref{A1}) it follows
\begin{equation}
\int_{\Gamma}v_\eps\leq\frac{a_{4}}{c_{0}}|\Gamma| \label{A3}%
\end{equation}
and  (\ref{A2}) implies that $w_\eps\leq\frac{a_{4}a_{5}}{c_{0}a_{6}}$.
Moreover, since $w_\eps\geq0$, we obtain from (\ref{eq:AtDInf3}) that
\begin{equation}
\int_{\Gamma}U_\eps\leq m \label{A5}%
\end{equation}
Using (\ref{A2}) we can rewrite (\ref{eq:AtDInf2}) as
\begin{equation}
0=\eps \Delta v_\eps-\left(  a_{1}\varepsilon+\frac{a_{2}\varepsilon U_\eps}{a_{3}%
\varepsilon+U_\eps}+c\right)  v_\eps+%
\frac{a_{4}U_\eps}{\varepsilon+U_\eps}-a_{5}v_\eps+a_{5}\fint
_{\Gamma}v_\eps\,, \qquad  x\in\Gamma\,.\label{A6}%
\end{equation}

Let $\varphi\in C^{2}(\Gamma)  $ an arbitrary test
function. Multiplying (\ref{A6}) by $\varphi$ and integrating over
$\Gamma$ we obtain, after integrating by parts
\begin{equation}
0=\varepsilon\int_{\Gamma}v_\eps\Delta\varphi-\int_{\Gamma}\left(
a_{1}\varepsilon+\frac{a_{2}\varepsilon U_\eps}{a_{3}\varepsilon+U_\eps}+c\right)
v_\eps\varphi+a_{4}\int_{\Gamma}\frac{U_\eps}{\varepsilon+U_\eps}\varphi-a_{5}%
\int_{\Gamma}\varphi v_\eps+a_{5}\left(  \fint_{\Gamma}v_\eps\right)
\left(  \int_{\Gamma}\varphi\right)  \label{A7}%
\end{equation}
Notice that $0\leq\frac{U_\eps}{\varepsilon+U_\eps}<1$ and  $0\leq\frac
{U_\eps}{a_{3}\varepsilon+U_\eps}\leq1.$ Due to (\ref{A3}) and (\ref{A5}) we obtain that
there exists a subsequence $\varepsilon_{j}\rightarrow0$ of the given sequence such that
\begin{equation}
\ v_{\eps_j}\rightharpoonup {v}\,,\qquad U_{\eps_j}%
\rightharpoonup\baru\,, \qquad \frac{U_{\eps_j}}{\varepsilon_{j}+U_{\eps_j}}%
\overset{*}{\rightharpoonup}\xi\,\qquad \mbox{ as }j\rightarrow\infty\label{B2}%
\end{equation}
where ${v},\baru\in\mathcal{M}_{+}(\Gamma)$. The convergence of the sequences $(v_{\eps_j})_j$,
$(U_{\eps_j})_j$ takes place in the weak topology of measures and the convergence of $\frac{U_{\eps_j}}{\varepsilon_{j}+U_{\eps_j}}$ takes place in the weak$\ast$ topology of $L^{\infty}(\Gamma)$.
Taking the
limit in \eqref{A7} we obtain
\begin{align*}
0&=-\int_{\Gamma}c{v}\varphi+a_{4}\int_{\Gamma}\xi
\varphi-a_{5}\int_{\Gamma}\varphi {v}+a_{5}\left(  \fint
_{\Gamma} {v}\right)  \left(  \int_{\Gamma}\varphi\right)\\
\int_{\Gamma}\left(  c+a_{5}\right)  {v}\varphi&=a_{4}%
\int_{\Gamma}\varphi\xi+a_{5}\fint_{\Gamma} {v}%
\int_{\Gamma}\varphi
\end{align*}
where $\varphi$ is an arbitrary test function in $C^{2}(\Gamma)$. In particular this implies that $(c+a_{5}){v}$
is an absolutely continuous measure with can be written in terms of a
bounded density, more precisely
\begin{equation}
{v}=\frac{a_{4}\xi}{c+a_{5}}+\frac{a_{5}}{c+a_{5}}\fint_{\Gamma
}{v} \label{A8}%
\end{equation}

Letting $\alpha=a_5\fint_\Gamma v$ we arrive at
\begin{equation}
  {v}=\frac{a_{4}\xi}{c+a_{5}}+\frac{\alpha}{c+a_{5}} \,.\label{B1}%
\end{equation}
We now consider the limit of (\ref{eq:AtDInf1}). Multiplying this equation by
$\varphi$ where $\varphi\in C^{2}(\Gamma)  $
and using (\ref{B2}) we obtain the limit equation
\begin{equation*}
	\Delta\baru+c{v}-a_{4}\xi=0\,,
\end{equation*}
which is satisfied
in the sense of distributions.
Using (\ref{B1}) we obtain
\begin{equation}
\Delta\baru+\alpha g-a_{4}(1-g)  \xi=0\,.
\label{eqB3}
\end{equation}
Due to the boundedness of $g$ and $\xi$ classical regularity theory implies that $\baru$ is
bounded in $W^{2,p}(\Gamma)$ for any $p<\infty.$ This
gives (\ref{B4}). On the other hand, taking the limit of (\ref{eq:AtDInf3}), using
(\ref{B2}) and the uniform boundedness of $w_{\eps}$  we obtain (\ref{massU}).

It remains to prove $\xi u = u$. Using (\ref{eq:AtDInf1}), (\ref{A3}) and applying classical regularity arguments \cite[Proposition 4.3]{Stam64},
(see also \cite{GiMa12}),
we obtain uniform estimates for $U_{\eps_j}$ in $W^{1,q}(\Gamma)$ for all $1<q<2$.
Then $U_{\eps_j}\rightarrow\baru$ in $L^{q}(\Gamma)$ and for any test function $\varphi\in C^0(\Gamma)$
\begin{equation*}
	\int_\Gamma \varphi (\xi u-u) = \lim_{j\to\infty} \int_\Gamma \varphi \Big(\frac{U_{\eps_j}}{\eps_j+U_{\eps_j}}-1\Big)U_{\eps_j} = \lim_{j\to\infty} \int_\Gamma -\eps_j\varphi \frac{U_{\eps_j}}{\eps_j+U_{\eps_j}} = 0.
\end{equation*}
This implies $\xi u = u$.
\end{proof}

We remark that by an integration of \eqref{eqB3} first over $\Gamma$ and then over $\{u>0\}$ (and using the $W^{2,p}(\Gamma)$-regularity of $u$) we deduce that
\begin{equation}
	\alpha=\frac{a_{4}\int_{\Gamma}(1-g)  \xi}{\int_{\Gamma}g} = \frac{a_{4}\int_{\{u>0\}}(1-g)}{\int_{\{u>0\}}g}\,. \label{defAlpha}%
\end{equation}
Moreover, by Stampacchia's lemma \cite[Proposition 3.23]{GiMa12} and the $W^{2,p}(\Gamma)$-regularity of $u$ one obtains $\Delta u=0$ almost everywhere in $\{u=0\}$, which yields the representation formula
\begin{equation}\label{eq:repr-xi}
	\xi =
	\begin{cases}
		1 \quad &\text{ in }\{u>0\},\\
		\frac{\alpha g}{a_4(1-g)} \quad &\text{ in }\{u=0\}
	\end{cases}.
\end{equation}
In particular, from the second equality in \eqref{defAlpha} and \eqref{eq:repr-xi} we deduce that $\xi,\alpha$ are determined by $u$.

\subsection{Critical mass for polarization.}\label{Ss.criticalmass}

Without loss of generality, rescaling $u$ and $\alpha$ accordingly,  we set in the following $a_4=1$.
Thus, we consider a solution $(u,\xi,\alpha)$ of the problem
\begin{align}
 &-\Delta u = -(1-g)\xi + \alpha g\quad \mbox{ on } \Gamma,\qquad
 u\geq 0, \, u\in W^{2,p}(\Gamma)\,\text{ for all }1\leq p<\infty, \label{eq:u1}\\
 &\xi\in [0,1],\qquad \xi=1 \quad\text{ on }\{u>0\},\label{eq:xi1}\\
 &\int_{\Gamma} u = m \,,\label{eq:u2}
\end{align}
Here $g$ is as in \eqref{defg} and we also recall  the compatibility condition \eqref{defAlpha} for $\alpha$, which by \eqref{eq:xi1} in particular yields 
\begin{equation}\label{eq:alpha-ast}
	\alpha \leq \alpha_{\ast}:=\frac{\int_{\Gamma}(1-g)  }{\int_{\Gamma}g}.
\end{equation}

We have a variational principle for the problem \eqref{eq:u1}, \eqref{eq:xi1}.
\begin{proposition}\label{prop:var-inequality}
Fix $0\leq \alpha \leq\alpha_{\ast}$ and consider the minimization problem
\begin{equation}
	u=\arg\min\left\{  \tfrac{1}{2}\int_{\Gamma}|\nabla v|^2+\int_{\Gamma}(1-g)v - \alpha gv\;  |\, v \in H^1(\Gamma)\,, \, v\geq 0 \right\}. \label{D1}
\end{equation}
Then the following properties hold.
\begin{enumerate}
\item	For any $0\leq \alpha \leq\alpha_{\ast}$ a minimizer of \eqref{D1} exists.
\item	Minimizers are unique for $\alpha<\alpha_{\ast}$ and unique up to additive constants for $\alpha=\alpha_{\ast}$.
\item	$u\in H^1(\Gamma)$ with $u\geq 0$ is a minimizer of \eqref{D1} if and only if $u$ is a solution of the variational inequality
\begin{equation}
	\int_\Gamma \nabla u\cdot\nabla (u-v) + \int_\Gamma \big((1-g)-\alpha g\big)(u-v) \leq 0 \quad\text{ for all } v\in H^1(\Gamma),\, v\geq 0. \label{D1a}
\end{equation}
\item $u\in H^1(\Gamma)$ with $u\geq 0$ is a minimizer in \eqref{D1} if and only if
\begin{equation}
	-\Delta u + (1-g) -\alpha g \geq 0\quad\text{ in }\Gamma,\qquad
	-\Delta u + (1-g) -\alpha g = 0 \quad\text{ in } \{u>0\}. \label{eq:D1b}
\end{equation}
\item	If $u\in W^{2,p}(\Gamma)$, $1\leq p<\infty$ with $u\geq 0$ solves \eqref{eq:u1}, \eqref{eq:xi1} then $u$ is a minimizer of \eqref{D1}.
\item	If $u\in W^{2,p}(\Gamma)$, $1\leq p<\infty$ is a minimizer of \eqref{D1} then there exists $\xi\in L^\infty(\Gamma)$ with $0\leq\xi\leq 1$ such that \eqref{eq:u1}, \eqref{eq:xi1} are satisfied.
\end{enumerate}
\end{proposition}
\begin{proof}
\begin{enumerate}
\item
We use the direct method of the calculus of variations. For any $v \in H^1(\Gamma)$, $v\geq 0$ we have, letting $\bar v = \frac{1}{|\Gamma|}\int_\Gamma v$,
\begin{align}
	I(v) &:= \tfrac{1}{2}\int_{\Gamma}|\nabla v|^2+\int_{\Gamma}(1-g)v - \alpha gv \nonumber\\
	&= \tfrac{1}{2}\int_\Gamma |\nabla (v-\bar v)|^2 + \int_\Gamma \big((1-g)-\alpha g\big)(v-\bar v) + \bar v\int_\Gamma \big((1-g)-\alpha g\big)\nonumber\\
	&\geq C\|v-\bar v\|_{H^1(\Gamma)}^2 -C(\alpha,g,\Gamma)\|v-\bar v\|_{L^2(\Gamma)} + \bar v (\alpha_{\ast}-\alpha)\int_\Gamma g \label{coercivity}\\
	& \geq C\|v-\bar v\|_{H^1(\Gamma)}^2 - C(\alpha,g,\Gamma) + \bar v (\alpha_{\ast}-\alpha)\int_\Gamma g. \nonumber
\end{align}
If $0\leq \alpha<\alpha_{\ast}$ this shows the coercivity of $I$ and then also the existence of a minimizer. (We remark that for $\alpha > \alpha_{\ast}$
we can take a sequence $v_n=n$ such that the functional converges to $-\infty$.)

In the case $\alpha=\alpha_{\ast}$ we have $I(v)=I(v+\gamma)$ for any $\gamma\in\R$. We then can minimize $I$
in the class $\{v\in H^1(\Gamma)\,:\, \int_\Gamma v =0\}$ and obtain a minimizer $v_\ast$ for this constrained minimization problem, which in addition satisfies,
since $I(v)=I(v-\bar v)$,
\begin{equation*}
	I(v_\ast) \leq \inf\left\{  I(v)\;  |\, v \in H^1(\Gamma)\,, \, v\geq 0 \right\}.
\end{equation*}
Moreover, $v_\ast$ solves the Euler--Lagrange equation $-\Delta v_\ast = -(1-g)+\alpha_\ast g+\lambda$, for some Lagrange multiplier $\lambda$ (we even find $\lambda=0$ by integration over $\Gamma$). By elliptic regularity, $v_\ast$ is bounded. But then $u:= v_\ast -\min_\Gamma v_\ast$ is nonnegative and satisfies \eqref{D1}.
\item
Consider two different minimizer $u_1,u_2$ of \eqref{D1}. We deduce that for $v=\frac{1}{2}(u_1+u_2)$
\begin{align*}
	0\leq I(v)-\frac{I(u_1)+I(u_2)}{2} = -\tfrac 1 8 \int_\Gamma |\nabla u_1-\nabla u_2|^2,
\end{align*}
hence $u_1-u_2$ is constant. In the case $\alpha<\alpha_{\ast}$ we further obtain from $I(u_1)=I(u_2)$, using the second line in \eqref{coercivity},  that $u_1=u_2$.
\item
\eqref{D1a} is the weak form of the Euler--Lagrange inequality and holds for any solution of \eqref{D1}. Vice versa, for a nonnegative solution $u\in H^1(\Gamma)$ of \eqref{D1a} and any other $v\in H^1(\Gamma)$, $v\geq 0$ we deduce from \eqref{D1a} that
\begin{equation*}
	I(u)-I(v) \leq \int_\Gamma \frac{1}{2}|\nabla u|^2 - \frac{1}{2}|\nabla v|^2 - \nabla u\cdot\nabla (u-v) = -\frac{1}{2}\int_\Gamma |\nabla (u-v)|^2\leq 0.
\end{equation*}
Therefore, $u$ is a minimizer in \eqref{D1}.
\item
\eqref{D1a} implies that $u$ is a weak solution of \eqref{eq:D1b}.
Vice versa, if $u\in H^1(\Gamma)$, $u\geq 0$ satisfies \eqref{eq:D1b} and $v\in H^1(\Gamma)$, $v\geq 0$ then
\begin{align*}
	&\int_\Gamma \nabla u\cdot\nabla (u-v) + \int_\Gamma \big((1-g)-\alpha g\big)(u-v) \\
	&\qquad\qquad = \int_\Gamma \big(\nabla u\cdot\nabla (u-v)_+  +\big((1-g)-\alpha g\big)(u-v)_+\big)\\
	&\qquad\qquad\qquad - \int_\Gamma \big(\nabla u\cdot\nabla (u-v)_- + \big((1-g)-\alpha g\big)(u-v)_-\big)
\end{align*}
Since $u>0$ in $\{u>v\}$ the first integral on the right hand side  is zero by the equation in \eqref{eq:D1b}. The inequality in \eqref{eq:D1b} gives that the second integral is nonpositive. Hence, \eqref{D1a} holds.
\item \eqref{eq:u1}, \eqref{eq:xi1} immediately yield \eqref{eq:D1b}.
\item Consider a solution $u\in W^{2,p}(\Gamma)$, $u\geq 0$ of \eqref{D1}. We then can choose $\xi=1$ in $\{u>0\}$ and $\xi = \frac{\alpha g}{1-g}$ in $\{u=0\}$. Using that by Stampacchia's lemma $\Delta u=0$ almost everywhere in $\{u=0\}$ we deduce first from the inequality in \eqref{eq:D1b} that $\xi\in [0,1]$, and second that \eqref{eq:u1} holds.
\end{enumerate}
\end{proof}
We have not yet shown that the assumption $u\in W^{2,p}(\Gamma)$, $1\leq p<\infty$ for a minimizer $u$ of \eqref{D1} that appears in item 6
of the previous proposition is always satisfied. This however will follow from Proposition \ref{prop:var-inequality} together with  Theorem \ref{CritMass} below.

\begin{remark}\label{rem:maxmon}
  The variational inequality \eqref{D1a} is equivalent to the following variational inequality without constraints: $u\in H^1(\Gamma)$ is a solution of \eqref{D1a} with $u\geq 0$ if and only if
  \begin{equation}
    \int_\Gamma \nabla u\cdot \nabla (u-v) + (1-g)u_+ -(1-g)v_+ \leq \int_\Gamma \alpha g (u-v) \quad\text{ for all }v\in H^1(\Gamma). \label{eq:var-ineq2}
  \end{equation}
  Note in particular that the latter property implies $u\geq 0$ by choosing $v=u_+$. In fact, this gives $\int_{\{u<0\}} |\nabla u|^2\leq 0$ which implies
  $0=\int_{\{u<0\}} |\nabla u|^2 \leq -\int_{\gamma} \alpha g  u_{-} \leq 0$ and hence $u_-=0$.

  Introducing the functional $J\colon H^1(\Gamma)\to\R$,
  \begin{equation*}
    J(v) = \int_\Gamma (1-g)v_+
  \end{equation*}
  \eqref{eq:var-ineq2} is equivalent to
  \begin{equation}
    -\Delta u + \partial J(u) \ni \alpha g. \label{eq:maxmon}
  \end{equation}
  We further obtain that $\partial J(u)=(1-g)\partial p(u)$, for $p:\R\to\R$, $p(r)=r_+$. In particular, for $u,\xi$ as in \eqref{eq:u1}, \eqref{eq:xi1} we find $(1-g)\xi\in\partial J(u)$.
  Hence \eqref{eq:u1}, \eqref{eq:xi1} is equivalent to the partial differential inclusion \eqref{eq:maxmon} for a maximal monotone graph.
\end{remark}
We will say that a solution to \eqref{eq:u1}, \eqref{eq:xi1}
exhibits polarization if the measure of both of the two regions $\{  u>0\}
$ and $\{u=0\}  $ is positive. It turns out that it is possible
to give a necessary and sufficient condition for polarization in this model.
To this end we  define an auxiliary function
$u_{\ast}.$

\begin{definition}
\label{defWstar} Let $g$ be
defined as in (\ref{defg}) and $\alpha_\ast$ as in \eqref{eq:alpha-ast}. We define
$u_{\ast}$ as the unique solution of the equation
\begin{equation}
  -\Delta u_{\ast}=-(1-g)  +\alpha_{\ast}g\,,\qquad  x\in\Gamma\,\label{D2}%
\end{equation}
such that $\min_{\Gamma}u_{\ast}=0.$
\end{definition}
Equation (\ref{D2}) has infinitely many solutions due to the
definition of $\alpha_{\ast},$ which implies $\int_{\Gamma}\left[
-(1-g)  +\alpha_{\ast}g\right]  =0.$ The solutions
differ by a constant and they are continuous in $\Gamma.$ Then, if $u$
is any solution of (\ref{D2}) we obtain that $u_{\ast}=u-\min_{\Gamma
}u$ and $u_{\ast}$ is unique.
We define
\begin{equation}
  m_{\ast}= m_{\ast}[g]  =\int_{\Gamma}u_{\ast} \label{mCrit}%
\end{equation}

In the following we denote  $g_{\max}=\max_{\Gamma}g$
and $\alcrit= \frac{(1-g_{\max})}{g_{\max}}$. Notice that $\alcrit\leq \alpha_{\ast}$ and strict inequality holds if $g$ is not constant.
In  case  that $g$ is constant no polarization occurs. More precisely we have the following equivalence.
\begin{proposition}\label{pro:constant_g}
Consider a solution $(u,\xi,\alpha)$ of \eqref{eq:u1}-\eqref{eq:u2} with $m>0$. Then $u$ is constant if and only if $g$ is constant. In this case we have $u_*=0$ and $\alpha=\alcrit=\alpha_{\ast}$.
\end{proposition}
\begin{proof}
Let $u\geq 0$ be constant. Since $m>0$ we have $u>0$ and $\xi=1$. Then \eqref{eq:u1} implies that $g$ is constant. This yields $u_*=0$ and $\alpha=\alcrit=\alpha_{\ast}$.

Vice versa assume that $g$ is constant. By $m>0$ the set $\{u>0\}$ has positive measure and \eqref{defAlpha} yields that $\alpha=\frac{1-g}{g}$ and $\xi=1$. But then \eqref{eq:u1} implies $\Delta u=0$ and $u$ is constant.
\end{proof}
 Our main results concerning the onset of polarized states is contained in the following theorem, see also the remark below.
\begin{theorem}
\label{CritMass}
Let $g$ be as in (\ref{defg}) and $m_{\ast}$ as in (\ref{mCrit}).
\begin{enumerate}[a)]
 \item
For any $m>0$, there exists a unique solution $(u,\xi,\alpha)$ to \eqref{eq:u1}-\eqref{eq:u2}.
\item
For any $\alpha\in [0,\alpha_*]$ there exists a solution $(u,\xi)$ to \eqref{eq:u1}, \eqref{eq:xi1}.\\
 For $\alpha<\alpha_*$ this solution is unique. For $\alpha=\alpha_\ast$ we have $\xi=1$ and $u$ is unique up to an additive constant.\\
 The map $\alpha\mapsto (u,\xi)$ is monotone increasing.\\
 For $0\leq\alpha < \alcrit$ we have $u=0$. If $g$ is not constant this also holds for $\alpha=\alcrit$.
\item If $m>m_{\ast} $ the solution $(u,\xi,\alpha)$ of \eqref{eq:u1}-\eqref{eq:u2} satisfies $u>0$ and $\xi=1$ in $\Gamma$, $\alpha=\alpha_{\ast}$
and $u=u_{\ast}+m-m_{\ast}$.
\item If $m<m_{\ast} $  the solution $(u,\xi,\alpha)$ of \eqref{eq:u1}-\eqref{eq:u2} satisfies $\left\vert \{  u=0\}  \right\vert >0$
and $\alpha<\alpha_{\ast}$.
\end{enumerate}
\end{theorem}

\begin{proof}
 \begin{enumerate}
 \item
Solvability of  \eqref{eq:u1}-\eqref{eq:u2} for given $m>0$ follows from Theorem \ref{thm:stat-states-Dpositive}, Theorem \ref{thm:stat-states-D=infty} and Theorem \ref{Limeps}.

\item {\it Solvability of the variational inequality and monotonicity.}
The variational inequality \eqref{D1a} can be
solved for each $\alpha\leq\alpha_{\ast}$, see Proposition \ref{prop:var-inequality}, and yields a unique solution for $\alpha<\alpha_\ast$ and a solution that is unique up to additive constants if $\alpha=\alpha_\ast$.

Suppose that $\alpha_{1}< \alpha_{2}$ and let $(u_1,\alpha_1)$, $(u_2,\alpha_2)$ be
solutions of \eqref{D1a}. We use $v_1=\min\{u_1,u_2\}$ in \eqref{D1a} for $u_1$ and  $v_2=\max\{u_1,u_2\}$ in \eqref{D1a} for $u_2$  and observe that $u_1-v_1=(u_1-u_2)_+$, $u_2-v_2=-(u_1-u_2)_+$. Adding the resulting inequalities yields
\begin{equation}
\int_{\Gamma}|\nabla (u_{1}-u_{2})_{+}|^{2}+\int_{\Gamma}(\alpha_2-\alpha_1) g (u_1-u_2)_+ \leq 0\,. \label{D3}
\end{equation}
Hence, we find $u_1 \leq u_2$.

 \item
{\it Solutions of \eqref{D1a} are trivial for subcritical $\alpha$.}

We next show that solutions for $0\leq \alpha\leq \alcrit$ are trivial, unless $\alpha=\alcrit$ and $g$ is constant. Consider $\alpha= \alcrit$ and the unique solution $u$ of the variational inequality \eqref{D1a}. Assume that  $u\not\equiv 0$. We note that
\begin{equation*}
	(1-g)-\alcrit g \geq 1-g -\alcrit g_{\max} = 1-g - 1 +g_{\max} = g_{\max}- g \geq 0.
\end{equation*}
Using $v=0$ in  \eqref{D1a} yields
\begin{equation*}
	\int_\Gamma |\nabla u|^2 + \int_\Gamma (g_{\max}-g)u \leq \int_\Gamma |\nabla u|^2 + \int_\Gamma \big((1-g)-\alcrit g\big)u \leq 0,
\end{equation*}
which gives a contradiction, unless $g\equiv g_{\max}$ and $u$ is a constant.

Now for all $0\leq \alpha<\alcrit$ by the monotonicity property and comparing $u$ with the trivial solution for $\alpha=\alcrit$ we deduce that the solution of \eqref{D1a} vanishes identically, also in the case that $g$ is constant.

By Proposition \ref{prop:var-inequality} this proves the claim.

\item {\it Uniqueness of solution for given $m$.}
Consider $\alcrit\leq\alpha_1\leq \alpha_2$ and two solutions $(u_1,\alpha_1)$, $(u_2,\alpha_2)$ of the variational inequality \eqref{D1a}  with $\int_\Gamma u_1=\int_\Gamma u_2$. Then by monotonicity $u_1\leq u_2$ and the equality $\int_\Gamma (u_1-u_2)=0$ implies $u_1=u_2$. Assume now that $\alpha_1<\alpha_2$. Then $\alpha_2>\alcrit$ and $u_1=u_2\not\equiv 0$ by \eqref{D1a}. We then can use the equality in \eqref{eq:D1b} to deduce that $\alpha_1=\alpha_2$, a contradiction. Hence $(u_1,\alpha_1)=(u_2,\alpha_2)$.

This in particular shows that for any $m>0$ the solution $(u,\xi,\alpha)$ of \eqref{eq:u1}-\eqref{eq:u2} is unique (here we use that $\xi$ is determined by $u,\alpha$ through the representation formula \eqref{eq:repr-xi}).

Furthermore, consider the map $\alpha\mapsto m$ that assigns to $\alpha$ the total mass $m$ of the solution of the variational inequality \eqref{D1a}. By the monotonicity and uniqueness properties proved above $\alpha\mapsto m$ is strictly increasing on $[\alcrit,\alpha_\ast]$.

\item {\it Equivalence of \eqref{eq:u1},\eqref{eq:xi1} and \eqref{D1a}.}
From \eqref{eq:alpha-ast} we already know that $\alpha \leq \alpha_{\ast}$ is necessary for any solution of \eqref{eq:u1},\eqref{eq:xi1}.

We now remove the additional regularity assumption in Proposition \ref{prop:var-inequality} item 6 and show that a solution $(u,\alpha)$ of the variational inequality \eqref{D1a} solves \eqref{eq:u1},\eqref{eq:xi1} for a suitable $\xi$ (determined by \eqref{eq:repr-xi}). In fact let  $(u,\alpha)$ be a solution of \eqref{D1a}.
Then there exists a solution $(\tilde u,\xi,\tilde \alpha)$ of \eqref{eq:u1}-\eqref{eq:u2} with $m=\int_\Gamma u$. In particular, $(\tilde u,\tilde\alpha)$ solves \eqref{D1a} and the uniqueness property above shows that $(u,\alpha)=(\tilde u,\tilde \alpha)$. Therefore $(u,\xi,\alpha)$ solve \eqref{eq:u1},\eqref{eq:xi1}.

This equivalence together with the monotonicity property shown above proves that $m\mapsto (
\alpha,u)$ is strictly monotone. Finally, by \eqref{eq:repr-xi} then also the map $m\mapsto \xi$ is monotone.

\item {\it Proof of items c) and d).}
Therefore, for each $m\in\left[0,m_{\ast}\right) $ we obtain a unique solution $(u,\xi,\alpha)$, $\alpha\in\left[\alcrit,\alpha_{\ast}\right) $ of \eqref{eq:u1}-\eqref{eq:u2}.
For $\alpha<\alpha_{\ast}$ we deduce from \eqref{defAlpha} that $\xi<1$ in a set of positive measure. Therefore $|\{u=0\}|>0$.

If $m\geq m_{\ast}  $ we must have necessarily $\alpha
=\alpha_{\ast}.$  Since $\alpha=\frac{\int_{\Gamma}(1-g) \xi}{\int_{\Gamma}g}$ we obtain that $\xi =1$ for a.e. $x \in \Gamma$,
whence $u$ solves (\ref{D2}). Thus $u$ and $u_{\ast}$ only differ by a constant $A$ and using (\ref{mCrit}) we obtain $A=m-m_{\ast}$, whence item c) in Theorem \ref{CritMass} follows.
\end{enumerate}

\end{proof}

\begin{remark}
The previous theorem shows that polarization, in the sense specified above, occurs if $0<m<m_*$ and that there is no polarization for $m>m_*$.

In the case $m=m_*$ polarization may occur or not. We give an example for a function $g$ that yields a polarized solution $u$ with $m=m_*$, i.e.~with $\left\vert \{u_*=0\}\right\vert >0$:

Fix $\alpha_*=\frac{1}{2}$ and choose any $u_*\in C^2(\Gamma)$ with $u_*\geq 0$, $|\{u_*=0\}|>0$, and $|\Delta u_*|<\frac{1}{2}$. We define
\begin{equation*}
	g := \frac{1-\Delta u_*}{1+\frac{1}{2}},
\end{equation*}
then $g$ is continuous with values in $(0,1)$. Moreover, by construction
\begin{equation*}
	-\Delta u_* = \frac{1}{2}g -(1-g),\qquad \alpha_* = \frac{\int_\Gamma (1-g)}{\int_\Gamma g}.
\end{equation*}

A sufficient condition for the property $|\{u_*=0\}|=0$ (i.e.~no polarization occurs) for $m=m_*$ is
\begin{equation*}
	\left|\left\{g=\frac{1}{1+\alpha_*}\right\}\right| = 0.
\end{equation*}
In fact, almost everywhere in $\{u_*=0\}$ we have $\Delta u_*=0$ and hence $g=\frac{1}{1+\alpha_*}$.
\end{remark}

\subsection{Localization of $\{u>0\}$ if $m\rightarrow 0.$}\label{Ss.localization}
By the results of the previous section, for non-constant $g$ and sufficiently small mass polarization occurs.
We can say a bit more about the patterns that arise if we consider the limit of vanishing mass: in a well-defined sense the polarized region concentrates on the set where $g$ (and hence also $c$) takes its maximal value.
To give a precise statement, let $S=\{ x\in \Gamma\,|\, g(x)=g_{\max}\}$ and $\alcrit=\frac{1-g_{\max}}{g_{\max}}$ as above.

\begin{proposition}\label{prop:localization}
 Consider a sequence $m_n\searrow 0$ and denote the corresponding solution to \eqref{eq:u1}-\eqref{eq:u2} by $(u_n,\xi_n,\alpha_n)$. Then the following holds
 \begin{enumerate}
 \item $\alpha_n \searrow \alcrit$.
 \item Assume that $u_n$ attains a maximum in $x_n$. Then $g(x_n) \to g_{\max}$.
 \item $\|u_n\|_{W^{2,p}(\Gamma)} \to 0$ for all $1\leq p<\infty$.
 \item Let $\Omega_n:=\{u_n >0\}$. Then
\begin{equation}\label{eq:omegan}
 \frac{1}{|\Omega_n|} \int_{\Omega_n} \Big(1-\frac{g}{g_{\max}}\Big) = \frac{\alpha_n - \alcrit}{1+\alpha_n} \to 0 \qquad \mbox{ as } n \to \infty.
\end{equation}
\item For any $\delta>0$ there exists $n_0$ such that if $\dist(x,S)>\delta$ then $u_n(x)=0$ for all $n \geq n_0$.
\end{enumerate}
\end{proposition}
\begin{proof}
 \begin{enumerate}
 \item We multiply \eqref{eq:u1} by a positive test function $\varphi$ and integrate by parts to obtain
 \[
  0 = \int_{\Gamma} u_n \Delta \varphi - \int_{\Gamma} (1-g)\xi_n\varphi + \alpha_n \int_{\Gamma} g\varphi\,.
 \]
 Now $|\int_{\Gamma} u_n \Delta \varphi |\leq C m_n \to 0$. Hence $\limsup_{n \to \infty} \alpha_n \leq \frac{\int_{\Gamma} (1-g)\varphi}{\int_{\Gamma} g \varphi }$. We now take a sequence $(\varphi_\ell)_\ell$ of test functions
 such that $\spt(\varphi_\ell)\subset \{g>g_{\max}-\frac{1}{\ell}\}$.
\item Since $m>0$ we have $u_n(x_n)>0$ and $u_n>0$ in some neighborhood of $x_n$. Then $\xi_n=1$ in this neighborhood and $u_n$ is $C^2$-regular in $x_n$. We therefore have  $-\Delta u_n(x_n)\geq 0$, $\xi_n(x_n)=1$ and  hence
 $0 \leq -  (1-g(x_n))+\alpha_n g(x_n)$, which implies $\alpha_n \geq \frac{ 1-g(x_n)}{g(x_n)} \geq \alcrit$. Since $\alpha_n \to \alcrit$ it follows that $g(x_n) \to g_{\max}$.

 \item
 By \eqref{B4}, $\alpha\leq \alpha_*$ and since $(m_n)_n$ is bounded we deduce that $(u_n)_n$ is uniformly bounded in $W^{2,p}(\Gamma)$ for any $1\leq p<\infty$. Since $u_n \to 0$ in $L^1(\Gamma)$ the result follows by interpolation.

 \item
 Integrate \eqref{eq:u1} over $\Omega_n$ and recall that $\xi_n=1$ in $\Omega_n$. Since $u_n \in W^{2,p}(\Gamma)$ we have $\int_{\Omega_n} \Delta u_n=0$ and we find
 \begin{align*}
  0&= - \int_{\Omega_n} (1-g) + \alpha_n \int_{\Omega_n}g\\
  &= -|\Omega_n| + (1+\alpha_n) \int_{\Omega_n} (g-g_{\max}) + g_{\max}(1+\alpha_n) |\Omega_n|\\
  &= -(1+\alpha_n) \int_{\Omega_n} (g_{\max}-g)   + (\alpha_n - \alcrit) g_{\max}|\Omega_n|
 \end{align*}
 and \eqref{eq:omegan} follows.

 \item
It is sufficient to prove the claim for all $0<\delta<\delta_0$, where $\delta_0>0$ is an arbitrary fixed number, to be chosen below. By continuity of $g$ there exists $\theta>0$, $\theta=\theta(\delta_0)$, such that
\begin{equation*}
	g<g_{\max}-\theta\quad\text{ in }\{x\in\Gamma\,:\,\dist(x,S)>r\},\quad r=\frac{\delta}{2}.
\end{equation*}
This implies, for sufficiently large $n$, that
\begin{align*}
	\alpha_n  g - (1-g) = \alpha_n g - (1-g) - \alpha_0 g_{\max} +(1-g_{\max}) &< (\alpha_n-\alpha_0)g_{\max} - (\alpha_n+1)\theta\\
	& \leq - \frac{2}{3}(\alpha_n+1)\theta
\end{align*}
in $\{x\in\Gamma\,:\,\dist(x,S)>r\}$.

Now fix an arbitrary $x_0\in\Gamma$ with $\dist(x_0,S)>\delta$ and assume $u_n(x_0)>0$. Then $g<g_{\max}-\theta$ on $D(x_0,r):=B(x_0,r)\cap \Gamma$.

Next define $\tilde u: D(x_0,r)\to\R$, $\tilde u(x):= \frac{(\alpha_n+1)\theta}{16}|x-x_0|^2$. For $\delta_0>0$ sufficiently small, only depending on the geometry of $\Gamma$, we then have
\begin{equation*}
	-\Delta\tilde u = \frac{(\alpha_n+1)\theta}{16}\big(-4 - 2\vec{H}(x)\cdot (x-x_0)\big) \geq -\frac{(\alpha_n+1)\theta}{2}\quad\text{ on }D(x_0,r),
\end{equation*}
where $\vec H$ denotes the mean curvature vector of $\Gamma$.
On the other hand, since $D(x_0,r)\subset \{\dist(\cdot,S)>r\}$
\begin{equation*}
	-\Delta u_n = \alpha_n g -(1-g) \leq - \frac{2}{3}(\alpha_n+1)\theta \quad\text{ on }D(x_0,r)\cap \{u_n>0\}.
\end{equation*}
In addition $0\leq u_n < \frac{(\alpha_n+1)\theta}{32}r^2$ for $n\geq n_0$ for some $n_0\in\N$ sufficiently large, only depending on $
\theta,r$ (thus only depending on $\delta$), since $u_n\to 0$ uniformly by item 3. Then
\begin{align*}
	-\Delta (u_n-\tilde u)& \leq -(\alpha_n+1)\theta + \frac{(\alpha_n+1)\theta}{2} <0\quad\text{ on }D(x_0,r)\cap \{u_n>0\},\\
	 u_n-\tilde u &\leq \frac{(\alpha_n+1)\theta}{32}r^2 - \frac{(\alpha_n+1)\theta}{16}r^2 <0 \quad\text{ on }\partial D(x_0,r)\cap \{u_n>0\},\\
	 u_n-\tilde u & <0 \quad\text{ on } D(x_0,r)\cap \partial \{u_n>0\}.
\end{align*}
By the weak maximum principle it follows that
\begin{equation*}
	u_n - \tilde u  \leq 0 \quad\text{ in } D(x_0,r)\cap \{u_n>0\},
\end{equation*}
which yields $u_n(x_0)\leq \tilde u_n(x_0)=0$, a contradiction.

This shows that $u_n(x_0)=0 $ for $n\geq n_0$, $n_0=n_0(\delta)$. Since $x_0\in \{\dist(\cdot,S)>\delta\}$ was arbitrary this proves the claim.
\end{enumerate}

\end{proof}

%
\section{Finite cytosolic diffusion $D<\infty$}\label{S.Dfinite}
In this section we now deal with the case of finite cytosolic diffusion $D<\infty$ and the fully coupled system. Using a similar
rescaling as in the case $D=\infty$ we derive an asymptotic limit in the form of a bulk-surface obstacle problem. In the case of spherical cell geometry
we characterize polarized states in terms of this limit problem and in terms of the total amount of proteins in the cell.

\subsection{Derivation of a nonlocal obstacle problem}\label{Ss.limeps}
For finite $D$ we use a similar rescaling of the steady-state equation \eqref{eq:org-stat-1}-\eqref{eq:org-stat-5} as in the previous section but consider in addition to \eqref{E2} that $D$ becomes large with $\eps\to 0$, more precisely $D\leadsto \frac{1}{\eps}D$ with $D$ of order one. This yields the following bulk-surface system.
\begin{align}
	0  &  =\Delta U_\eps+\left( \eps a_{1}+\eps\frac{a_{2}U_\eps}{a_{3}\eps%
+U_\eps}+c\right)  v_\eps-\frac{a_{4}U_\eps}{\eps+U_\eps} &&\text{ on }\Gamma\label{eq:eps-stat-1}\\
	0  &  =\eps\Delta v_\eps-\left(  a_{1}\eps+\eps\frac{a_{2}U_\eps}{a_{3}\eps%
+U_\eps}+c\right)  v_\eps+\frac{a_{4}U_\eps}{\eps+U_\eps}-a_{5}v_\eps+a_{6}w_\eps &&\text{ on }\Gamma\label{eq:eps-stat-2}\\
	0  &  =D\Delta w_\eps &&\text{ in } \Omega\label{eq:eps-stat-3}\\
-D\frac{\partial w_\eps}{\partial n}  &  =-a_{5}v_\eps+a_{6}w_\eps &&\text{ on }\Gamma\label{eq:eps-stat-4}.%
\end{align}
and
\begin{equation}
	\int_\Gamma (U_\eps+\eps v_\eps) + \int_B \eps w_{\eps} =m.  \label{eq:mass-eps}
\end{equation}
By Theorem \ref{thm:stat-states-Dpositive} for given $m>0$ there exists a nonnegative solution of this system. We first prove some uniform bounds.
\begin{theorem}\label{thm:bounds-eps-stat}
For any $0<m\leq m_0$ and any nonnegative solution $(w_{\eps},U_\eps,v_\eps)$ of \eqref{eq:eps-stat-1}-\eqref{eq:mass-eps} we have
\begin{equation}
	\|U_\eps\|_{H^2(\Gamma)} +\|v_\eps\|_{L^2(\Gamma)}+\|w_{\eps}\|_{H^1(\Omega)}  \leq  C(a_4,a_5,a_6,m_0,c,\Omega). \label{eq:bounds-eps-stat}
\end{equation}
\end{theorem}
\begin{proof}
Integrating \eqref{eq:eps-stat-1} over $\Gamma$ we deduce that
\begin{equation}
	\int_\Gamma v_\eps \leq \frac{1}{c_0}a_4|\Gamma|. \label{eq:ctrl-mass-veps}
\end{equation}
Similarly, by taking the sum of \eqref{eq:eps-stat-1} and \eqref{eq:eps-stat-2}, integrating and using \eqref{eq:ctrl-mass-veps} we also obtain
\begin{equation}
	\int_\Gamma w_\eps \leq \frac{1}{a_6c_0}a_4a_5|\Gamma|. \label{eq:ctrl-mass-weps}
\end{equation}

We next test \eqref{eq:eps-stat-2},\eqref{eq:eps-stat-3} with $a_5v_\eps$ and $a_6w_{\eps}$, respectively. Summing up the resulting equations gives
\begin{align}
	\int_\Gamma \Big(\eps a_5|\nabla v_\eps|^2+ ca_5v_\eps^2 \Big) + \int_\Omega a_6D|\nabla w_{\eps}|^2&\leq \int_\Gamma \Big( a_4a_5 v_\eps -(a_5v_\eps -a_6w_\eps)^2\Big) \notag\\
	&\leq C(a_4,a_5,c_0,\Omega), \label{eq:bounds-eps-stat-1}
\end{align}
and in particular a uniform $L^2(\Gamma)$-bound for $v_\eps$.

Applying the Poincar{\'e} inequality in $\{w\in H^1(\Omega):\int_\Gamma w =0\}$ we deduce that
\begin{align*}
	\|w_\eps\|_{L^2(\Omega)}^2 &= \Big\|w_\eps -\frac{1}{|\Gamma|}\int_\Gamma w_\eps\Big\|_{L^2(\Omega)}^2 + \frac{|\Omega|}{|\Gamma|^2}\Big(\int_\Gamma w_\eps\Big)^2\\
  &\leq C(\Omega)\|\nabla w_\eps\|_{L^2(\Omega)}^2+C(\Omega)\Big(\int_\Gamma w_\eps\Big)^2 \leq C(a_4,a_5,a_6,c_0,\Omega),
\end{align*}
and therefore obtain the required bound for $\|w_\eps\|_{H^1(\Omega)}$.

Finally, by $L^2$-regularity theory for \eqref{eq:eps-stat-1}, the mass bound \eqref{eq:mass-eps} and interpolation we conclude \eqref{eq:bounds-eps-stat}.
\end{proof}

With these uniform estimates we can pass to the limit $\eps\to 0$. We then obtain a generalized obstacle type problem.
\begin{theorem}\label{thm:obs1}
Consider a sequence $(w_{\eps},U_\eps,v_\eps)$ of nonnegative solutions to \eqref{eq:eps-stat-1}-\eqref{eq:eps-stat-4} with total mass $0<m\leq m_0$. Then there exists a subsequence $\eps\to 0$ and functions $(w,u,v)$
such that
\begin{align*}
	U_\eps &\weakto u\quad\text{ in }H^2(\Gamma)\,, 
	&v_\eps &\weakto v\quad\text{ in }L^2(\Gamma)\,, 
	&w_\eps &\weakto w\quad\text{ in }H^1(\Omega)\,. 
\end{align*}
Moreover there exists $\xi\in L^\infty(\Gamma)$ with $0\leq\xi\leq 1$ such that
\begin{align}
	0  &  =\Delta u+c v-a_4\xi &\text{ on }\Gamma,\label{eq:obs-1}\\
	0  &  =-c v+a_4\xi-a_{5}v+a_{6}w& \text{ on }\Gamma, \label{eq:obs-2}\\
	0  &  =D\Delta w &\text{ in } \Omega, \label{eq:obs-3}\\
	-D\frac{\partial w}{\partial n}  &  =-a_{5}v+a_{6}w&\text{ on }\Gamma  \label{eq:obs-4}%
\end{align}
and such that $u \xi =u$ and $\int_\Gamma u=m$ hold. Moreover, $w,u$ and $v$ are all nonnegative and $u\in W^{2,p}(\Gamma)$ for any $1\leq p<\infty$, $w\in C^\infty(\Omega)\cap C^0(\overline{\Omega})$, and $v\in L^\infty(\Gamma)$, with
\begin{equation}
	\|u\|_{W^{2,p}(\Gamma)}+\|w\|_{C^0(\overline{\Omega})} + \|v\|_{L^\infty(\Omega)} \leq C(p,a_4,a_5,a_6,D,\Omega,m_0). \label{eq:est-thm4.2}
\end{equation}
\end{theorem}
\begin{proof}
By the uniform bounds provided by Theorem \ref{thm:bounds-eps-stat} we obtain a subsequence and functions $w,u,v,\xi$ such that
\begin{alignat}{3}
	w_\eps &\weakto w \quad&&\text{in }H^1(\Omega),\nonumber\\
	U_\eps &\weakto u &&\text{in }H^2(\Gamma),\label{eq:conv-ueps} \\
	v_\eps &\weakto v &&\text{in }L^2(\Gamma),\nonumber\\
	\frac{U_\eps}{\eps+U_\eps} &\weakstarto \xi &&\text{in }L^\infty(\Gamma). \nonumber
\end{alignat}
In particular we obtain from Sobolev embedding that $U_\eps\to u$ in $L^2(\Gamma)$ and by the compactness of the trace map $H^1(\Omega)\embeds L^2(\Gamma)$ that $w_\eps\to w$ in $L^2(\Gamma)$. We then pass easily in the weak formulations of  \eqref{eq:eps-stat-1}-\eqref{eq:eps-stat-4} to the limit and deduce \eqref{eq:obs-1}-\eqref{eq:obs-4}. Moreover, we have for any $\eta\in C^0(\Gamma)$
\[	\int_\Gamma \eta (\xi u -u) = \lim_{\eps\to 0}\int_\Gamma \eta\big( \frac{U_\eps}{\eps+ U_\eps}U_\eps -U_\eps\big) = \lim_{\eps\to 0}\int_\Gamma \eta\frac{\eps U_\eps}{\eps+U_\eps} = 0,
\]
which proves $\xi u=u$. Finally by the uniform bounds \eqref{eq:bounds-eps-stat}, \eqref{eq:ctrl-mass-veps}  on $v_\eps,w_\eps$ we have
\[	\int_\Gamma \eps v_\eps + \int_\Omega \eps w_\eps \to 0
\]
and \eqref{eq:mass-eps}, \eqref{eq:conv-ueps} yield that $m=\int_\Gamma u$.

Since $0\leq \frac{U_\eps}{\eps+U_\eps}\leq 1$ the corresponding bounds for $\xi$ follow. Moreover, by \eqref{eq:bounds-eps-stat} and \eqref{eq:conv-ueps} we deduce that
\begin{equation*}
	\|u\|_{H^2(\Gamma)} +\|v\|_{L^2(\Gamma)}+\|w\|_{H^1(\Omega)}  \leq  C(a_4,a_5,a_6,m_0,c,\Omega).
\end{equation*}

For $p>2$  we test \eqref{eq:obs-3} with $(\kappa_pw)^{p-1}$, $\kappa_p:=\frac{a_6}{a_5}$ as well as  \eqref{eq:obs-2} with $v^{p-1}$. This yields
\begin{align*}
	0 &= -\int_\Omega D(p-1)\kappa_p^{p-1} w^{p-2}|\nabla w|^2 +\int_\Gamma\Big( (a_5v-a_6w)\big((\kappa_p w)^{p-1}-v^{p-1}\big) -cv^p + a_4\xi v^{p-1}\Big)\\
	&\leq -\int_\Gamma \Big(a_5 (v-\kappa_pw)\big[v^{p-1}-(\kappa_pw)^{p-1}\big] +cv^p - a_4v^{p-1}\Big)
	\,\leq\, -\frac{c_0}{2}\int_\Gamma v^p + C(c_0,a_4,p,\Gamma),
\end{align*}
and hence $v$ is bounded in $L^p(\Gamma)$ for any $1\leq p<\infty$.

By \cite[Theorem 3.14]{Nitt11} we then obtain that $w\in C^{0,\gamma}(\Omega)$ for some $\gamma>0$, with
\begin{equation}
	\|w\|_{C^{0,\gamma}(\Omega)} \leq C(\Omega,c_0)a_5\|v\|_{L^3(\Gamma)}, \label{eq:w-C0gamma}
\end{equation}
which in particular proves the desired maximum bound for $w$. As a harmonic function, $w$ is smooth in the open set $\Omega$.

The maximum bound for $w$ and \eqref{eq:obs-2} yield $v\in L^\infty(\Gamma)$, with a corresponding bound.

By standard $L^p$-regularity for \eqref{eq:obs-1} we finally deduce that $u\in W^{2,p}(\Gamma)$ for any $1\leq p<\infty$ and
\begin{equation*}
	\|u\|_{W^{2,p}(\Omega)} \leq  C(a_4,a_5,a_6,m_0,c,\Omega,p).
\end{equation*}
\end{proof}
We can express the system above as an equation for $u$ only, but involving a non-local (pseudo-differential) operator given by the Neumann to Dirichlet operator $T$ applied to the Laplace--Beltrami operator on $\Gamma$. For properties of the Neumann to Dirichlet (and the Dirichlet to Neumann) operators see the appendix.
\begin{proposition}\label{prop:obs00}
Let $(u,v,w,\xi)$ be nonnegative functions with the same regularity as in Theorem \ref{thm:obs1}. Then the following statements are equivalent.
\begin{enumerate}
\item
$(u,v,w,\xi)$ satisfies \eqref{eq:obs-1}-\eqref{eq:obs-4}.
\item
$(u,w,\xi)$ satisfies
\begin{align}
	0 & = \Delta u -a_4(1-g)\xi + a_6gw, \qquad u\xi = u\quad\text{ a.e.~on }\Gamma, \label{eq:obs00-var1}\\
	0 & =\Delta w \text{ in } \Omega, \qquad
	D\frac{\partial w}{\partial n}  = a_4(1-g)\xi - a_6gw \text{ on }\Gamma  \label{eq:obs00-var2}%
\end{align}
and $v$ is given by
\begin{equation}
		v = \frac{1-g}{a_5}(a_6w+a_4\xi). \label{eq:obs-var-v}
\end{equation}
\item
There exists $\bar w\in\R$ such that
\begin{equation}
	0 = \Delta u -a_4(1-g)\xi + a_6g\big(\bar w +\frac{1}{D}T\Delta u\big), \qquad u\xi = u\quad\text{ a.e.~on }\Gamma, \label{eq:obs0-var1}
\end{equation}
and $v,w$ are given by \eqref{eq:obs-var-v} and
\begin{align}
	w &= \bar w +\frac{1}{D}T\Delta u, \label{eq:obs0-var-w}
\end{align}
\end{enumerate}
Moreover, $\bar w$ is determined by
\begin{equation}
	\bar w = \fint_\Gamma w = \frac{1}{a_6\int_\Gamma g} \int_\Gamma \Big(a_4(1-g)\xi - \frac{a_6}{D}gT\Delta u\Big).	\label{eq:obs0-var-barw}
\end{equation}
\end{proposition}
\begin{proof}
We first observe that \eqref{eq:obs-2} is equivalent to \eqref{eq:obs-var-v}. Using this we see that \eqref{eq:obs-1}, \eqref{eq:obs-3}, \eqref{eq:obs-4} are equivalent to \eqref{eq:obs00-var1}, \eqref{eq:obs00-var2}. We then have that \eqref{eq:obs00-var1}, \eqref{eq:obs00-var2} is equivalent to \eqref{eq:obs00-var1} and
\begin{equation*}
	D\Delta w = 0\quad\text{ in }\Omega,\qquad D\frac{\partial w}{\partial n} = \Delta_\Gamma u,
\end{equation*}
which gives the equivalence to \eqref{eq:obs0-var1}, \eqref{eq:obs0-var-w}. Finally, \eqref{eq:obs0-var-barw} follows by integrating \eqref{eq:obs0-var1} over $\Gamma$.
\end{proof}

\subsection{Localization of $\{u>0\}$ if $m\rightarrow 0.$}\label{Ss.localization-1}
For simplicity we set $a_4=1$ in the following. As above let $S=\{ x\in \Gamma\,|\, g(x)=g_{\max}\}$ and $\alpha_0=\frac{1-g_{\max}}{g_{\max}}$.
\begin{proposition}\label{prop:localization-1}
Consider a sequence $m_k\to 0$ of positive numbers and for $k\in\N$ any nonnegative solution $(u_k,v_k,w_k,\xi_k)$ of \eqref{eq:obs-1}-\eqref{eq:obs-4} with $\int_\Gamma u_k=m_k$. Then the following holds
 \begin{enumerate}
 \item $\|u_k\|_{W^{2,p}(\Gamma)} \to 0$ for any $1\leq p<\infty$.
 \item $\overline{w}_k \to \frac{1}{a_6}\alcrit$.
 \item Assume that $u_k$ attains a maximum in $x_k$. Then $g(x_k) \to g_{\max}$.
 \item Let $\Omega_k:=\{u_k >0\}$. Then
\begin{equation*}
 \lim_{k\to\infty}\frac{1}{|\Omega_k|} \int_{\Omega_k} \Big(1-\frac{g}{g_{\max}}\Big) =  0 \qquad \mbox{ as } k \to \infty.
\end{equation*}
\item For any $\delta>0$ there exists $k_0$ such that if $\mbox{dist}(x,S)>\delta$ then $u_k(x)=0$ for all $k \geq k_0$.
\end{enumerate}
\end{proposition}
\begin{proof}
 \begin{enumerate}
\item By \eqref{eq:est-thm4.2} the sequence $(u_k)_k$ is uniformly bounded in $W^{2,p}(\Gamma)$ for any $1\leq p<\infty$.
Using that $u_k\to 0$ in $L^1(\Gamma)$ interpolation gives $u_k\to 0$ in any $W^{2,p}(\Gamma)$, $1\leq p<\infty$, hence also in $C^{1,\gamma}(\Gamma)$ for any $0\leq \gamma<1$.

By \eqref{eq:n9}, the Neumann to Dirichlet operator $T:L^p(\Gamma)\to W^{1,p}(\Gamma)$ is continuous for any $1<p<\infty$. Using \eqref{eq:obs0-var-w} this implies that $w_k-\overline{w}_k\to 0$ in every $W^{1,p}(\Gamma)$.
 \item We multiply \eqref{eq:obs00-var1} by a positive test function $\varphi$ and integrate by parts to obtain
 \begin{align*}
  0 &= \int_{\Gamma} u_k \Delta \varphi -  \big((1-g)\xi_k- a_6gw_k\big) \varphi \\
  &\geq \int_{\Gamma} u_k \Delta \varphi - \int_{\Gamma} \big((1-g)- a_6g(w_k-\overline{w}_k)\big)\varphi +a_6\overline{w}_k \int g\varphi
 \end{align*}
Using $m_k\to 0$ and the properties from the first item, we deduce that
 $\limsup_{k \to \infty} \overline{w}_k \leq \frac{\int_{\Gamma} (1-g)\varphi}{a_6 \int_{\Gamma} g \varphi }$. We now take a sequence $(\varphi_\ell)_\ell$ of test functions
 such that $\spt(\varphi_\ell)\subset \{g>g_{\max}-\frac{1}{\ell}\}$ and obtain
\begin{equation*}
 	\limsup_{k \to \infty} \overline{w}_k \leq \frac{1}{a_6}\alcrit.
\end{equation*}
 \item Since $u_k$ is locally $C^2$-regular in $\{u_k>0\}$ we have  $-\Delta u_k(x_k)\geq 0$ and $\xi(x_k)=1$ and  hence
\begin{equation*}
	0 \leq -  (1-g(x_k))+a_6g(x_k)\big(\overline{w}_k - (w_k-\overline{w}_k)(x_k)\big).
\end{equation*}
It follows together with item 2 of this proof that
\begin{equation*}
	\frac{1}{a_6}\alcrit \geq \limsup_{k\to\infty} \overline{w}_k \geq \liminf_{k\to\infty} \overline{w}_k \geq \liminf_{k\to\infty}  \frac{(1-g(x_k))}{a_6g(x_k)} \geq \frac{1}{a_6}\alcrit.
\end{equation*}
It follows that $\overline{w}_k\to \frac{1}{a_6}\alcrit$ and that $g(x_k) \to g_{\max}$.
\item
Integrate \eqref{eq:obs00-var1} over $\Omega_k$ and recall that $\xi_k=1$ in $\Omega_k$. Since $u_k\in W^{2,p}(\Gamma)$ we have $\int_{\Omega_k}\Delta u_k=0$. This yields
 \begin{align*}
  -\fint_{\Omega_k} a_6 g (w_k-\overline{w}_k)&= - \fint_{\Omega_k} (1-g) +  \overline{w}_k\fint_{\Omega_k} a_6 g\\
  &= -1 + (1+a_6\overline{w}_k) \fint_{\Omega_k} (g-g_{\max}) + g_{\max}(1+a_6\overline{w}_k) \\
  &=g_{\max}\Big[ (-\alcrit +a_6\overline{w}_k)   + (1+a_6\overline{w}_k) \fint_{\Omega_k} \frac{g-g_{\max}}{g_{\max}}\Big]
 \end{align*}
and the claim follows by item 2 and since $w_k-\overline{w}_k\to 0$ in $C^0(\Gamma)$.
\item
We argue as in Proposition \ref{prop:localization}. Consider $0<\delta<\delta_0$, where $\delta_0>0$ is chosen below. Then there exists $\theta>0$, $\theta=\theta(\delta_0)$, such that in $\{x\in\Gamma\,:\,\dist(x,S)>r\},\quad r=\frac{\delta}{2}$
\begin{equation*}
	g<g_{\max}-\theta,\qquad
	- (1-g) + \alpha_0 g \leq - (\alpha_0+1)\theta,
\end{equation*}
We deduce that in  $\{x\in\Gamma\,:\,\dist(x,S)>r\}$ for all $k\geq k_0$ sufficiently large
\begin{equation*}
	-(1-g) + a_6g\overline{w}_k \leq - (\alpha_0+1)\theta +a_6g\big(\overline{w}_k -\frac{1}{a_6}\alpha_0\big)\leq - \frac{1}{2}(\alpha_0+1)\theta,
\end{equation*}
where we have used item 2.

Moreover, by item 1, \eqref{eq:n9} and the Sobolev embedding $T\Delta u_k\to 0$ in $W^{1,p}(\Gamma)$ for all $1\leq p<\infty$ and uniformly on $\Gamma$. Then, possibly after increasing $k_0$, we deduce
\begin{equation*}
	-(1-g) + a_6g\big(\overline{w}_k + \frac{1}{D}T\Delta u_k\big) \leq -\frac{1}{4}(\alpha_0+1)
\end{equation*}
Now fix an arbitrary $x_0\in\Gamma$ with $\dist(x_0,S)>\delta$ and assume $u_k(x_0)>0$. We let $D(x_0,r)=B(x_0,r)\cap \Gamma$ and define $\tilde u: D(x_0,r)\to\R$,
$\tilde u(x):= \frac{(\alpha_0+1)\theta}{32}|x-x_0|^2$.
For $\delta_0>0$ sufficiently small, only depending on the geometry of $\Gamma$, we then have
\begin{equation*}
	-\Delta\tilde u \geq -\frac{(\alpha_0+1)\theta}{6}\quad\text{ on }D(x_0,r).
\end{equation*}
On the other hand, since $D(x_0,r)\subset \{\dist(\cdot,S)>r\}$
\begin{equation*}
	-\Delta u_k \leq - \frac{1}{4}(\alpha_0+1)\theta \quad\text{ on }D(x_0,r)\cap \{u_k>0\}.
\end{equation*}
In addition $0\leq u_k < \frac{(\alpha_0+1)\theta}{32}r^2$ for $k\geq k_0$ for $k_0\in\N$ sufficiently large, only depending on $\delta$, since $u_k\to 0$ uniformly by item 3. Then as in Proposition \ref{prop:localization}  the weak maximum principle yields  a contradiction.
\end{enumerate}

\end{proof}

%
%
\subsection{Refined analysis for the case of spherical cell shapes}\label{Ss.existenceobs}
In the following we restrict ourselves to the case of a spherical cell and let $\Omega=B(0,1)\subset\R^3$. This allows us to characterize the system  \eqref{eq:obs-1}-\eqref{eq:obs-4} as a generalized obstacle problem that involves the Dirichlet to Neumann operator $N$ (see the appendix for a definition).
In order to reduce the number of parameters we define $\ell=\frac{a_6}{D}$.
\begin{proposition}\label{prop:obs}
Let $(u,v,w,\xi)$ be nonnegative functions with the same regularity as in Theorem \ref{thm:obs1}. Then $(u,v,w,\xi)$ satisfies \eqref{eq:obs-1}-\eqref{eq:obs-4} if and only if $u$ solves for some $\alpha\in\R$
\begin{align}
 0=\Delta u -(1-g)\xi + \alpha g - g\ell \big[N(u)+(u-\overline{u})\big]\,,\qquad u\geq 0\,, \label{eq:obs-var1a1}\\
 0\leq\xi\leq 1,\qquad u\xi = u \text{ almost everywhere on }\Gamma\,, \label{eq:obs-var1a2}
\end{align}
and if in addition $w$ is given by $a_6 w= \alpha -\ell g\big[N(u)+(u-\overline{u})\big]$ and $v$ by \eqref{eq:obs-var-v}.
The number $\alpha$ is then determined by
\begin{equation}
	\alpha = \frac{1}{\int_\Gamma g} \int_\Gamma \Big((1-g)\xi + \ell g\big[N(u)+(u-\overline{u})\big]\Big). \label{eq:def-alpha-new}
\end{equation}
\end{proposition}
\begin{proof}
By Proposition \ref{prop:obs00} (see in particular \eqref{eq:obs0-var1}) it is sufficient to prove that for $\Omega=B(0,1)$ and any $u\in H^2(\Gamma)$ we have $T\Delta_\Gamma u = -Nu -(u-\overline{u})$.

Let $F$ denote the solution of
\begin{equation*}
	\Delta F = 0\quad\text{ in }\Omega,\qquad F = u\quad\text{ on }\Gamma,
\end{equation*}
hence $\partial_n F=Nu$ on $S^2$, and let $G$ denote the solution of
\begin{equation*}
	\Delta G = 0\quad\text{ in }\Omega,\qquad \frac{\partial}{\partial n}G = \Delta_\Gamma u\quad\text{ on }\Gamma,\qquad \int_\Gamma G=0,
\end{equation*}
hence $G=T\Delta_\Gamma u$ on $S^2$.

We now consider the position vector field, $\eta(x)=x$ for $x\in \overline{B(0,1)}$. We observe that
\begin{equation*}
	0 = \Delta\big( G +\eta\cdot\nabla F + F\big)\quad\text{ in }B(0,1)
\end{equation*}
and that on $\Gamma=S^2$
\begin{align*}
	 \partial_n\big( G +\eta\cdot\nabla F + F\big) &= \frac{\partial}{\partial n}G + n\cdot D^2F n + 2n\cdot\nabla F \\
	 &= \Delta_\Gamma u + \partial_r^2 F + 2\partial_r F  = \Delta F=0,
\end{align*}
where in the last step we have used that $u=F$ on $\Gamma$ and that $\Delta  = \partial^2_r+\frac{2}{r}\partial_r + \frac{1}{r^2}\Delta_{S^2}$ in spherical coordinates. Hence $G +\eta\cdot\nabla F + F$ is constant and since $\int_\Gamma G=\int_\Gamma \eta\cdot\nabla F=0$ we deduce
\begin{equation*}
	T\Delta_\Gamma u = G= -\eta\cdot \nabla F - (F-\bar F) = -Nu -(u-\overline{u}),
\end{equation*}
which proves the claim.
\end{proof}

In the sequel we do not make any further use of the particular geometry, the following analysis applies to any solution of \eqref{eq:obs-var1a1}, \eqref{eq:obs-var1a2} and any $\Omega,\Gamma=\partial\Omega$ as before. However, equivalence between the systems \eqref{eq:obs-1}-\eqref{eq:obs-4} and \eqref{eq:obs-var1a1}, \eqref{eq:obs-var1a2} requires spherical symmetry.
\begin{proposition}
Let $(u,\xi,\alpha)$ be a solution of \eqref{eq:obs-var1a1}, \eqref{eq:obs-var1a2}. Then for almost every $x\in\Gamma$
\begin{equation}
	\xi(x) =
	\begin{cases}
		1 \quad&\text{ if } u(x)>0,\\
		\frac{g(\alpha - \ell [N(u)+(u-\overline{u})])}{1-g}(x) \quad&\text{ if }u(x)=0.
	\end{cases}
	\label{eq:prop:xi}
\end{equation}
\end{proposition}
\begin{proof}
Let $\Sigma=\{ u=0\}$. Then by Stampacchia's Lemma $\nabla u=0$ in $\Sigma$ and $D^2u=0$ almost everywhere in $\{\nabla u=0\}\supset \Sigma$. In particular, by the pointwise almost everywhere equality \eqref{eq:obs-var1a1} we deduce the characterization of $\xi$ in $\Sigma$. On the other hand, by $u\xi=u$ we conclude that $\xi=1$ in $\Gamma\setminus \Sigma$.
\end{proof}

\begin{proposition}[Monotonicity]\label{prop:monotonicity}
 Let $(u_1,\xi_1,\alpha_1)$, $(u_2,\xi_2,\alpha_2)$ be two solutions of \eqref{eq:obs-var1a1}, \eqref{eq:obs-var1a2}.
 \begin{enumerate}
 	\item If $\alpha_1 < \alpha_2$ then $u_1\leq u_2$ holds. Moreover, we have $\int_\Gamma u_1<\int_\Gamma u_2$ unless $u_1=u_2=0$.
 	\item If $\alpha_1=\alpha_2$ then either $u_1=u_2$ or $u_1,u_2$ differ by a nonzero constant. In the latter case $\xi_1=\xi_2=1$ holds. 
 \end{enumerate}
\end{proposition}
\begin{proof}
Let $\alpha_1\leq \alpha_2$. The difference $U:=u_1-u_2$ satisfies
\begin{equation}
	0 = \Delta U -(1-g)(\xi_1-\xi_2) + g(\alpha_1-\alpha_2) -g\ell \big[N(U)+ (U-\overline{U})\big].	\label{eq:maxprinc-1}
\end{equation}
We assume that the maximum of $U$ is positive and argue by a maximum principle.
By \cite{Bony67} (see also \cite{Lion83}) for any maximum point $x_0$ of $U$
there exists a sequence $(x_k)_k$ in $\Gamma$ with $x_k\to x_0$ and $\lim_{k\to\infty} \Delta U(x_k)\leq 0$. By assumption and continuity of $U$ we can assume without loss of generality that $U(x_k)>0$ for all $k\in\N$. Hence $\xi_1(x_k)=1\geq \xi_2(x_k)$ and the second term in \eqref{eq:maxprinc-1} as well as the third term are nonpositive in $x_k$. We next consider the term involving the Dirichlet to Neumann operator: Let $F$ be the solution of
\begin{equation*}
	\Delta F=0\quad\text{ in }\Omega,\qquad F|_\Gamma=U,
\end{equation*}
hence $N(U)=\nabla F\cdot\nu$. By assumption and the weak maximum principle
$F$ attains a positive maximum in $x_0$ and by the Hopf boundary point lemma and the strong maximum principle we have $N(U)(x_0)=(\nabla F\cdot\nu)(x_0)>0$  unless $U$ is constant. By \eqref{eq:N-W1p} the Dirichlet to Neumann map $N:W^{2-\frac{1}{p},p}(\Gamma)\to W^{1-\frac{1}{p},p}(\Gamma)$ is continuous for any $1<p<\infty$. By Sobolev embedding we deduce that $N(U)$ is continuous and, unless $U$ is constant,
\begin{equation*}
	\lim_{k\to\infty} - g(x_k)N(U)(x_k) <0.
\end{equation*}
Finally, in $x_0$ the term $U-\overline{U}$ is also nonnegative. Therefore, from \eqref{eq:maxprinc-1} we deduce a contradiction unless $U$
is a positive constant. Hence, $u_1\leq u_2$ or $u_1=u_2+\gamma$ for some positive constant $\gamma$.

\begin{enumerate}
\item
First assume $\alpha_1<\alpha_2$. If $U=\gamma$ is positive then $\xi_1\geq \xi_2$ and \eqref{eq:maxprinc-1} gives a contradiction.
Therefore $u_1\leq u_2$ holds. We now assume that $\int_{\Gamma}u_1=\int_{\Gamma}u_2$. This immediately gives $u_1=u_2$.
If now $\Sigma:=\{u_1>0\}=\{u_2>0\}$ is not empty we obtain by integrating \eqref{eq:obs-var1a1}, for $u_1$ and for $u_2$, over $\Sigma$ and deduce,
since $\int_{\Gamma}N(u)=0$, that
\begin{equation*}
	\alpha_1\int_\Sigma g - \alpha_2\int_\Sigma g = 0,
\end{equation*}
hence $\alpha_1=\alpha_2$, a contradiction.
\item
Next consider $\alpha_1=\alpha_2$. If $U$ is not constant then $u_1\leq u_2$, but interchanging the roles of $u_1,u_2$ also gives $u_2\leq u_1$ and hence equality.

On the other hand, if $U$ is constant, then \eqref{eq:maxprinc-1} implies that $\xi_1=\xi_2$. But now $u_1>0$ or $u_2>0$ and we deduce from \eqref{eq:obs-var1a2} that $\xi_1=\xi_2=1$.
\end{enumerate}
\end{proof}

\begin{corollary}[Existence, uniqueness and monotonicity]\label{cor:unique-Dfin}
For any $m>0$ there exists exactly one solution $(u,\xi,\alpha)$ of \eqref{eq:obs-var1a1}, \eqref{eq:obs-var1a2} with $\int_\Gamma u=m$. Moreover, the map $m\mapsto \alpha$ is monotone increasing. 
\end{corollary}
\begin{proof}
Fix $m>0$. By Theorem \ref{thm:stat-states-Dpositive}, Theorem \ref{thm:obs1} and Proposition \ref{prop:obs} there exists a solution $(u,\xi,\alpha)$ of \eqref{eq:obs-var1a1}, \eqref{eq:obs-var1a2} with $\int_\Gamma u =m$. Consider two solutions $(u_1,\xi_1,\alpha_1)$, $(u_2,\xi_2,\alpha_2)$ with $\int_\Gamma u_1=\int_\Gamma u_2=m$. Without loss of generality we may assume that $\alpha_1\leq \alpha_2$. Proposition \ref{prop:monotonicity} implies that $u_1 \leq u_2$ and since $\int_\Gamma (u_1-u_2)=0$ that $u_1=u_2$. By our assumption $m>0$ Proposition \ref{prop:monotonicity} further gives $\alpha_1=\alpha_2$, which by \eqref{eq:prop:xi} finally shows that $(u_1,\xi_1,\alpha_1)=(u_2,\xi_2,\alpha_2)$.

Again by Proposition \ref{prop:monotonicity}, $\alpha$ is a monotonically increasing function of $m$.
\end{proof}

We draw some first conclusions on the polarization properties of the model.
\begin{corollary}
Let  $u\in W^{2,p}(\Gamma)$ for any $1\leq p<\infty$ and $\xi\in L^\infty(\Gamma)$ with $0\leq\xi\leq 1$, $\xi u =u$ on $\Gamma$ be given such that \eqref{eq:obs-var1a1}, \eqref{eq:obs-var1a2} are satisfied.
Assume $m:=\int_\Gamma u >0$. Then $u$ is constant if and only if $g$ is constant. In that case we have $\xi=1$ on $\Gamma$.
\end{corollary}
\begin{proof}
Assume $g$ is constant and let $\overline{\xi} =\fint_\Gamma \xi$.
Testing \eqref{eq:obs-var1a1} with $u-\overline{u}$, using that $g$ constant implies $\int_\Gamma \alpha g(u-\overline{u})=0$, \eqref{eq:npos} yields
\begin{align*}
	0 &= \int_\Gamma |\nabla u|^2 + (1-g)\int_\Gamma \xi(u-\overline{u})+ g\ell \int_\Gamma \big[N(u)+(u-\overline{u})\big](u-\overline{u})\\
	&\geq  \int_\Gamma |\nabla u|^2 + (1-g)\int_\Gamma (u-\overline{\xi} u)
	= \int_\Gamma |\nabla u|^2 + (1-g)(1-\overline{\xi})\int_\Gamma u
\end{align*}
Since $\xi\leq 1$ we deduce that $u$ is constant and $\xi=1$.

Assume vice versa that $u$ is constant. Then $u>0$ and $\xi=1$, moreover $N(u)+(u-\overline{u})=0$. Then \eqref{eq:obs-var1a1} immediately gives that also $g$ is constant.
\end{proof}

%
%
\subsection{Spherical shell shapes: Existence of a critical mass for polarization.}\label{Ss.critmass2}
Also in the present case of a finite cytosolic diffusion rate we would like to give criteria for the onset of polarized states, i.e.~the occurrence of configurations $u$ such that both $\{u=0\}$ and $\{u>0\}$ have positive measure.

As in the case $D=\infty$ we first look for solutions $(u,\xi,\alpha)$ of \eqref{eq:obs-var1a1}, \eqref{eq:obs-var1a2} such that $u>0$. Then $\xi=1$ and
\begin{align}
 0=\Delta u -(1-g) + \alpha g - \ell g\big[N(u)+(u-\overline{u})\big]\,, \label{E2a}
\end{align}
for some suitable $\alpha$. We will prove that there exists a unique value $\alpha$ for which (\ref{E2a}) can be solved and characterize this critical value $\alpha=\alpha_*$. In the following we let $\tilde N$ denote the mapping $u\mapsto N(u)+(u-\overline{u})$. We then write the problem as
\begin{equation}
	-\Delta u+\ell g\tilde N(u)  =\alpha g-(1-g)  \label{E3}%
\end{equation}
We claim that the operator $-\Delta  +\ell g\tilde N$ has a simple zero eigenvalue (alternatively, the kernel is
one-dimensional). Indeed, it is clear that the constants are elements of the
kernel of this operator. Vice versa, any element $U$ of the kernel of $-\Delta  +\ell g\tilde N$ satisfies \eqref{eq:maxprinc-1} with $\alpha_1=\alpha_2$ and $\xi_1=\xi_2$, hence the arguments in the proof of Proposition \ref{prop:monotonicity} imply that $U$ is constant.

We proceed to characterize the solvability of \eqref{E3} by using Fredholm theory. We first obtain by \eqref{eq:N-Lp} that
\begin{equation*}
	\Big|\int_\Gamma gN(\varphi)\varphi\Big| \leq C(\max_\Gamma g) \|\varphi\|_{L^2(\Gamma)}\|\varphi\|_{H^1(\Gamma)},\qquad
	\Big|\int_\Gamma (\varphi-\bar \varphi)\varphi\Big| \leq C(\Omega)\int_\Gamma\varphi^2,
\end{equation*}
which implies the estimate
\begin{align*}
	\big\langle -\Delta\varphi +\ell g\tilde N(\varphi),\varphi\big\rangle &\geq \int_\Gamma |\nabla \varphi|^2 -C\ell (\max_\Gamma g)\Big(\|\varphi\|_{L^2(\Gamma)}\|\varphi\|_{H^1(\Gamma)} +\|\varphi\|_{L^2(\Gamma)}^2\Big)\\
	&\geq \frac{1}{2}\int_\Gamma |\nabla \varphi|^2 -C(\ell,g,\Omega)\|\varphi\|_{L^2(\Gamma)}^2.
\end{align*}
Therefore, there exists $\lambda\geq 0$ such that the operator $-\Delta + \ell g\tilde N + \lambda\id$ is coercive on $H^1(\Gamma)$ and, by the Lax--Milgram Theorem, bijective and continuous from $H^1(\Gamma)$ to $H^1(\Gamma)^*$ with continuous inverse.

Therefore, $K_\lambda:= (-\Delta + \ell g\tilde N + \lambda\id)^{-1}$ defines a compact operator from $L^2(\Gamma)$ to $L^2(\Gamma)$. For an arbitrary $f\in L^2(\Gamma)$ we deduce from \cite[Theorem D.5]{Evan10} first that
\begin{equation}
	-\Delta u + \ell g\tilde N(u) = f
\end{equation}
is solvable if and only if
\begin{equation}
	K_\lambda f \in \mathcal{R}(\id- \lambda K_\lambda) = \mathcal{N}(\id - \lambda K_\lambda^*)^\perp,
\end{equation}
and second that $\mathcal{N}(\id - \lambda K_\lambda^*)$ is one-dimensional. We further compute $K_\lambda^*=(-\Delta + \ell (g\tilde N)^* + \lambda\id)^{-1}$ and claim that $(g\tilde N)^*=\tilde N(g\cdot)$. In fact, for any $\varphi,\psi\in L^2(\Gamma)$ we deduce by \eqref{eq:Nsym} that
\begin{equation*}
	\int_\Gamma g\psi  \big[N(\varphi)+(\varphi-\bar\varphi)\big] = \int_\Gamma N(g\psi)\varphi +g\psi\varphi - \frac{1}{|\Gamma|}\int_\Gamma g\psi\int_\Gamma\varphi = \int_\Gamma \big[N(g\psi) +g\psi-\overline{g\psi}\big]\varphi.
\end{equation*}

Similarly as above, we then see that $\psi\in \mathcal{N}(\id-\lambda K_\lambda^*)$ if and only if
\begin{equation}
	-\Delta\psi + \ell \tilde N(g\psi) = 0. \label{eq:ker-adjoint}
\end{equation}
In order to characterize the critical value of $\alpha$, below we need that $\int\psi g\neq 0$.
\begin{lemma}\label{lem:adjoint}
Let $\psi\in L^2(\Gamma)$ be a nontrivial solution of \eqref{eq:ker-adjoint}. Then $\psi\in H^1(\Gamma)$ and either $\psi\geq 0$ or $\psi\leq 0$ on $\Gamma$.
\end{lemma}
\begin{proof}
By \eqref{eq:ker-adjoint} and \eqref{eq:N-H-1} we first obtain from $\psi\in L^2(\Gamma)$ that $\psi\in H^1(\Gamma)$.
Hence $\psi\in L^p(\Gamma)$ for any $1\leq p<\infty$, and a further application of \eqref{eq:N-H-1} implies $\psi\in W^{1,p}(\Gamma)$ and in particular $g\psi\in C^0(\Gamma)$. Hence there exists a unique classical solution $F\in C^\infty(\Omega)\cap C^0(\overline{\Omega})$  of
\begin{equation*}
	\Delta F = 0\quad\text{ in }\Omega,\qquad F= g\psi\text{ on }\Gamma.
\end{equation*}
In particular, we have $N(g\psi)=\partial_nF$.

We now integrate $\Delta F$ over the set $\{F>0\}\cap\Omega$. We remark that $\nabla F \cdot \nu_{\{F>0\}}\leq 0$ on $\Omega\cap\partial\{F>0\}$ and that $\nabla \psi \cdot \nu_{\{\psi>0\}}\leq 0$ on $\partial\{\psi>0\}$. Therefore, the Gauss Theorem yields
\begin{align*}
	0 = \int_{\{F>0\}} \Delta F &= \int_{\partial\{F>0\}\cap\Omega} \nabla F \cdot \nu_{\{F>0\}}\,d\Ha^2 + \int_{\{F>0\}\cap\Gamma}\partial_n F \notag\\
	&\leq  \int_{\{\psi>0\}} N(g\psi) = \frac{1}{\ell} \int_{\{\psi>0\}} \Delta \psi - \int_{\{\psi>0\}}\big(g\psi -\overline{g\psi}\big)\\
	&  \leq - \int_{\{\psi>0\}} g\psi +\frac{|\{\psi>0\}|}{|\Gamma|}\int_{\Gamma} g\psi \leq \Big(-1+\frac{|\{\psi>0\}|}{|\Gamma|}\Big)\int_{\{\psi>0\}} g\psi \leq 0.
\end{align*}
We therefore deduce that $|\{\psi>0\}|$ is either zero or coincides with $|\Gamma|$, hence $\psi\leq 0$ on $\Gamma$ or $\psi\geq 0$ on $\Gamma$.
\end{proof}

In the following let us fix any nontrivial nonnegative solution $\psi$ of \eqref{eq:ker-adjoint}. By the argument from above a solution of \eqref{E2a} exists if and only if
\begin{equation*}
	\int_{\Gamma}\psi[\alpha g-(1-g)]  =0,
\end{equation*}
whence if
\begin{equation}
	\alpha=\alpha_*:= \frac{\int_{\Gamma}\psi(1-g)  }{\int_{\Gamma}\psi g}. \label{E4}%
\end{equation}
(Note that $\alpha_*$ is well-defined by Lemma \ref{lem:adjoint}.) Then all solutions $(u,\alpha)$ of \eqref{E2a} have $\alpha=\alpha_*$ and can be written as $u=u_*+c$, where $c$ is constant and $u_*$ is the unique solution with
\begin{equation*}
	\min_{\Gamma} u_* = 0.
\end{equation*}

We now argue exactly as in the case
$\ell =0$. We define the critical mass as
\begin{equation}
	m_*:=\int_{\Gamma}u_* \label{eq:mCrit-D}
\end{equation}
and obtain the following characterization for the onset of polarized states.
\begin{theorem}\label{thm:CritMass-D}
Let $g$ be as in (\ref{defg}). Consider $m_*$ as defined in \eqref{eq:mCrit-D} and for $m>0$ the solution $(u,\xi,\alpha)$ of \eqref{eq:obs-var1a1}, \eqref{eq:obs-var1a2}.
\begin{enumerate}
\item[a)] The map $m\mapsto \alpha$ is strictly monotone increasing on $(0,m_\ast]$. If $m>0$ then $\alpha>\alpha_0$.
\item[b)] If $m>m_{\ast}$ we have that $u>0$ in $\Gamma$
and $\alpha=\alpha_{\ast}.$ Moreover
$u=u_{\ast}+(m-m_{\ast})\frac{1}{|\Gamma|}$.
\item[c)] If $m<m_{\ast}$ we have $\left\vert \{  u=0\}  \right\vert >0$
and $\alpha<\alpha_{\ast}$.
\end{enumerate}
\end{theorem}
\begin{proof}
\begin{enumerate}
\item[a)] We know from Corollary \ref{cor:unique-Dfin} that $m\mapsto \alpha$ is monotone increasing. Consider $m_1<m_2$ and assume that $\alpha_1=\alpha_2$. Since $u_1=u_2$ gives a contradiction Proposition \ref{prop:monotonicity} yields that $u_1$, $u_2$ only differ by a nonzero constant and  that $\xi_1=\xi_2=1$.
But then $u_1,u_2$ are solutions of \eqref{E2a}, which implies $\alpha_1=\alpha_2=\alpha_*$ and $m_1\geq m_*$, $m_2>m_*$. Therefore, $m\mapsto \alpha$ is strictly monotone increasing on $(0,m_*]$.

To prove the second claim, let $\xi_0:=\frac{\alpha_0 g}{1-g}$. Then $0\leq\xi_0\leq \frac{\alpha_0 g_{\max}}{1-g_{\max}}=1$, hence $(0,\xi_0,\alpha_0)$ is a solution of \eqref{eq:obs-var1a1}, \eqref{eq:obs-var1a2}. By monotonicity there are no solutions with $\alpha\leq \alpha_0$ and $\int_\Gamma u>0$.

\item[b)] Assume first that $\alpha<\alpha_*$. By Proposition \ref{prop:monotonicity} we deduce that $u\leq u_*$, hence $m\leq m_*$, a contradiction. Next assume that $\alpha>\alpha_*$. Again by monotonicity we have $u\geq u_*+c$ for any constant $c>0$, which again is impossible. Hence $\alpha=\alpha_*$ and $u=u_*+c$. We further conclude that $c|\Gamma|=\int_\Gamma (u-u_*)=m-m_*$ which gives the claimed characterization of $u$.
\item[c)] By part a) for any solution $(u,\xi,\alpha)$ with $m<m_*$ we have $\alpha\leq \alpha_*$. If $\alpha=\alpha_*$ then we would have $u=u_*+c$ with $c<0$, a contradiction to $u\geq 0$. This shows $\alpha<\alpha_*$.
Finally, if $u>0$ almost everywhere then $\xi=1$ almost everywhere and we deduce that $u$ is a solution of \eqref{E2a}. But this requires $\alpha=\alpha_*$.
\end{enumerate}
\end{proof}

\bigskip
{\bf Acknowledgment.}\\
We thank Waldemar Kolanus for several interesting discussions on cell polarization and  Danielle Hilhorst for bringing reference  \cite{BrNi82} to our attention. MR gratefully acknowledges the hospitality of the Hausdorff Center  of Mathematics at the University of Bonn.

\begin{appendix}\label{app:ND}
\section{The Dirichlet to Neumann operator}
Here we briefly introduce the Dirichlet to Neumann and the Neumann to Dirichlet operators and state the properties that have been used in our arguments. We set $H^{-\frac{1}{2}}_*(\Gamma):=\{f\in H^{-\frac{1}{2}}(\Gamma)\,:\, \langle f,1\rangle =0\}$, the subspace $H^{\frac{1}{2}}_*(\Gamma)$ is defined correspondingly. With a slight abuse of notation we write in the following $\int_\Gamma fg$ for $f\in H^{-\frac{1}{2}}(\Gamma)$, $g\in H^{\frac{1}{2}}(\Gamma)$ and the duality product $\langle f,g\rangle$.

For $f\in H^{\frac{1}{2}}(\Gamma)$, let $F\in H^1(\Omega)$ denote the weak solution of the Dirichlet problem
\begin{equation*}
	\Delta F =0  \quad\text{ in }\Omega,\qquad F=f \quad\text{ on }\Gamma
\end{equation*}
and let $Nf:=\frac{\partial}{\partial n}F$ be the Neumann data on $\Gamma$. This defines a continuous map $N:H^{\frac{1}{2}}(\Gamma)\to H^{-\frac{1}{2}}_*(\Gamma)$. 

Similarly, for $f\in H^{-\frac{1}{2}}_*(\Gamma)$, let $\tilde F\in H^1(\Omega)$ denote the weak solution of the Neumann problem
\begin{equation*}
	\Delta \tilde F =0  \quad\text{ in }\Omega,\qquad \frac{\partial}{\partial n}\tilde F =f \quad\text{ on }\Gamma,\quad \int_\Gamma f =0
\end{equation*}
and let $Tf:=\tilde F$ be the Dirichlet data on $\Gamma$. This defines a continuous map $T:H^{-\frac{1}{2}}_*(\Gamma)\to H^{\frac{1}{2}}_*(\Gamma)$. We extend $T$ to a continuous map
$T:H^{-\frac{1}{2}}(\Gamma)\to H^{\frac{1}{2}}_*(\Gamma)$ by setting $Tc=0$ for any constant function $c$.

\begin{proposition}\label{prop:Neumann_map}
The following properties hold.
\begin{align}
	\int_\Gamma (Nf)g &= \int_\Gamma fN g&&\text{ for all }f,g\in H^{\frac{1}{2}}(\Gamma)\label{eq:Nsym}\\%
	\int_\Gamma (Nf)f &= \int_\Omega |\nabla F|^2\quad&&\text{ for any }\varphi\in H^{\frac{1}{2}}(\Gamma)\label{eq:npos}\\
	N:W^{2-\frac{1}{p},p}(\Gamma)&\to W^{1-\frac{1}{p},p}(\Gamma)\quad&&\text{ is continuous for any }1<p<\infty, \label{eq:N-W1p}\\
  N:W^{1,p}(\Gamma)&\to L^p(\Gamma)\quad&&\text{ is continuous for any }1<p<\infty, \label{eq:N-Lp}\\
  N:L^p(\Gamma)&\to W^{-1,p}(\Gamma)\quad&&\text{ is continuous for any }1<p<\infty, \label{eq:N-H-1}\\
	T:L^p(\Gamma)&\to W^{1,p}(\Gamma)\quad&&\text{ is continuous for any }1<p<\infty. \label{eq:n9}
\end{align}
\end{proposition}
\begin{proof}
The properties \eqref{eq:Nsym}, \eqref{eq:npos} follow easily from the definition of $N$ and $F$.
 To prove \eqref{eq:N-W1p} we first remark that the trace space of $W^{k,p}(\Omega)$ is $W^{k-\frac{1}{p},p}(\Gamma)$, see \cite[Theorem 3.67]{DeDe12}. Given $f\in W^{2-\frac{1}{p},p}(\Gamma)$ by \cite[Theorem 9.14,Theorem 9.15]{GiTr01} there exists a harmonic function $F\in W^{2,p}(\Omega)$ with boundary data $f$ and $\|F\|_{W^{2,p}(\Omega)}\leq C\|f\|_{W^{2-\frac{1}{p},p}(\Gamma)}$. By the continuity of the trace operator $\|Nf\|_{W^{1-\frac{1}{p},p}(\Gamma)}\leq C\|F\|_{W^{2,p}(\Omega)}$ and the conclusion follows.

We deduce \eqref{eq:N-Lp} from \cite[Proposition III.4.9]{Tayl00}. By duality we infer \eqref{eq:N-H-1}. In fact, let $q=\frac{p}{p-1}$. For an arbitrary $\psi\in W^{1,q}(\Gamma)$ we deduce by \eqref{eq:N-Lp} that
\begin{equation*}
  \int_\Gamma (N\varphi)\psi = \int_\Gamma \varphi (N\psi) \leq \|\varphi\|_{L^p(\Gamma)}\|N\psi\|_{L^q(\Gamma)}\leq C\|\varphi\|_{L^p(\Gamma)}\|\psi\|_{W^{1,q}(\Gamma)},
\end{equation*}
hence $\|N \varphi\|_{W^{1,q}(\Gamma)^*}\leq C\|\varphi\|_{L^p(\Gamma)}$.

For \eqref{eq:n9} see the proof of \cite[Proposition III.4.11]{Tayl00}.
\end{proof}

\end{appendix}


\end{document}